%% file: TrivialGruposVersionPreliminar.tex
  \documentclass[12pt,oneside]{amsart}

\usepackage{amsmath}
\usepackage{amssymb}
\usepackage{multirow}
\usepackage{float}
\usepackage{graphicx}

\usepackage{palatino, url, multicol}
\usepackage[all,knot,arc]{xy}

\usepackage{subfig}

      \theoremstyle{plain}
      \newtheorem{thm}{Theorem}[section]
      \newtheorem{propst}[thm]{Proposition}
      \newtheorem{lemma}[thm]{Lemma}
      \newtheorem{cor}[thm]{Corollary}
      \newtheorem{conj}{Conjecture}
      \theoremstyle{definition}

      \numberwithin{equation}{section}
      
\newcommand\Q{{\mathbb{Q}}}

\title[Artin Presentations of the Trivial Group  and Hyperbolic Closed Pure $3$-Braids]{ Artin Presentations of the Trivial Group and Hyperbolic Closed Pure $3$-Braids}

\author[Armas]{Lorena Armas-Sanabria}
\address{Lorena Armas-Sanabria, UAEM-CONACYT}
\email{lorenaarmas089@gmail.com}
\author[Rodr\'{\i}guez]{Jes\'us Rodr\'{\i}guez Viorato}
\address{Jes\'us Rodr\'iguez Viorato, CONAHCyT CIMAT} 
\email{jesusr@cimat.mx}
\author[Jasso]{E. Fanny Jasso-Hern\'andez}
\address{E. Fanny Jasso-Hern\'andez, 
Mathematics Department, Virginia Tech, Blacksburg, VA 24061-0123, USA}
\email{efjasso@vt.edu}

\date{\today}

\begin{document}

\begin{abstract}
We consider a special class of framed links that arise from the hexatangle. Such links are introduced in \cite{AE09}, where it was also analyzed when the 3-manifold obtained after performing integral Dehn surgery on closed pure 3-braids is $S^3$. In the present paper, we analyze  the symmetries of the hexatangle and give a list of Artin $n$-presentations for the trivial group. These presentations correspond to the double-branched covers of the hexatangle that produce $S^3$ after Dehn surgery. Also, using a result of Birman and Menasco  \cite{BM94}, we determine which closed pure $3$-braids  are hyperbolic.
\end{abstract} 

\maketitle

\centerline{\it Dedicated to Professor Gonz\'alez-Acu\~na in his 80th anniversary}

\section{Introduction} \label{intro}

In \cite{AE09} closed pure 3-braids of the form 
$\hat \beta = \widehat {{\sigma_1^{2e_1}\sigma_2^{2f_1}({\sigma_2\sigma_1\sigma_2})^{2e}}}$ were considered; these are shown in Figure \ref{fig4}.
It was determined exactly when an integral surgery in such links produces the 3-sphere; to do it, we used the fact that such links are strongly invertible and that their exteriors are double-branched covers of certain fillings of what was called the hexatangle. Then, it was determined when the trivial knot is obtained by filling the hexatangle. These configurations giving the trivial knot are used in the present paper to obtain Artin $3$-presentations of the trivial group. 

An Artin $n$-presentation is a presentation of a group with generators \break$x_1,\cdots , x_n$
and relations $r_1, r_2,\cdots , r_n$
such that $\prod _{i=1} ^n r_ix_ir_i^{-1} = \prod_{i=1}^n x_i$ is satisfied in the free group  $F(x_1, x_2 \cdots x_n)$. Using Artin presentations Gonz\'alez Acu\~na \cite{GA75} characterized 3-manifold groups. In fact, he  proved that a group $G$ is the fundamental group of a closed, orientable 3-manifold if and only if $G$ admits an Artin n-presentation for some $n$. 

It is interesting to consider Artin $n$-presentations of the trivial group since there are  two conjectures (now theorems) relating them with the Poincar\'e conjecture (proved by 
Perelman). These conjectures were given by Gonz\'alez-Acu\~na \cite{GA75}, \cite{GA00}. 
 We now state such conjectures, but first, we give a  preliminary introduction taken from \cite{GA75}, \cite{GA00}.

Define $S_n$  as follows, $S_n = |x_1,y_1, \cdots,x_n,y_n: \prod _{i=1}^n x_i = \prod _{i=1}^n y_i^{-1}x_iy_i |$ and let $N\subset S_n$ be the normal closure  of $y_1,\cdots ,y_n$ in $S_n$. If $A$ is an Artin $n$-presentation define the automorphism $\phi _A: S_n \rightarrow S_n$ by $\phi _A (x_i) = r_i x_i r_i^{-1}, \phi _A (y_i) = r_iy_i \hskip 5pt i=1,\cdots , n$.

Two Artin $n$-presentations $A, A'$ are equivalent if there exist automorphisms $E_1, E_2$ of $S_n$ such that $E_1 \circ \phi _A = \phi _{A'}\circ E_2$ and $E_i(N) = N, \hskip 5pt i=1,2$.

One can see that $A\sim A'$ implies $M_A$ is homeomorphic to $M_{A'}$, which implies that $|A| \cong |A'|$, where $M_A, M_{A'}$
 are the 3-manifolds whose fundamental groups are given by $A, A'$ respectively. $|A|$ and $|A'|$, denotes the presentations of the groups. Then  the Poincar\'e conjecture is equivalent to
 
\begin{conj}
If $A$ is an Artin $n$-presentation
  such that $|A| = 1$ then $A\sim T$ where 
  $$T = ( x_1, \cdots ,x_n : x_1, \cdots , x_n )$$
\end{conj}
 
The knot groups in $S^3$ can be characterized. It follows from Artin's work that a group is the group of a knot if and only if it has a presentation:

 $$<y_1,\dots,y_n:s_1y_1s_1^{-1}y_2^{-1},\dots,s_{n-1}y_{n-1}s_{n-1}^{-1}y_n^{-1},s_ny_ns_n^{-1}y_1^{-1}>$$ such that 

\begin{equation}\label{eqn:artin-relation}
\quad \prod_{i=1} ^n s_iy_is_i^{-1}=\prod_{i=1}^{n} y_i  \quad  {\rm in} \quad F(y_1,\dots, y_n)
\end{equation}

If $G=<x_1,x_2\dots,x_n:r_1,r_2,\dots,r_n>$ is an Artin $n$-presentation of the trivial group, then the group given by the presentation   $G=<x_1,x_2\dots,x_n:r_1,r_2,\dots,r_{n-1}>$ is called a \textit{rat-group}. It is to say, a rat-group ( Reduced Artin Trivial ) is a group with a presentation of deficiency 1, obtained as follows: let $( x_1, x_2,...x_n; r_1, r_2, ...,r_n)$ be an Artin $n$-presentation of the trivial group. Then $(x_1,x_2,...,x_n ;r_1, r_2,...,r_{n-1})$ is a rat-group.
 The rat-groups are the knot groups in homotopic spheres.

\begin{conj}
A rat-group is a knot group, i.e., a rat-group has a presentation satisfying formula \ref{eqn:artin-relation}
\end{conj}

This conjecture is equivalent to the Poincar\'e conjecture. For its formulation, Gonz\'alez-Acu\~na used several characterizations of $S^3$ \cite{GA00}. 

Now, we  recall some well-known facts about $3$-manifolds $M$. A torus $T$  in $M$ is incompressible  if it does not have compression disks, i.e., embedded disks in  $M$ such that the boundary is in $T$  and is nontrivial in $T$. And $T$ is essential if it is incompressible and not boundary parallel  into $\partial M$.
Let $A$ be an annulus properly embedded in $M$. $A$ is essential if it is incompressible and not boundary parallel into $\partial M$.
Thurston proved the following: Let  $M$ be a compact, orientable, connected $3$-manifold with boundary. If $M$ does not have  either essential disks, spheres, annuli or tori then it is hyperbolic. And by Perelman, we know that
if $M$ is closed with no essential spheres or tori, and not Seifert fibered, then $M$ is hyperbolic.
A link  contained in $S^3$ is hyperbolic if its exterior is hyperbolic. A Dehn surgery on a hyperbolic link is exceptional if the 3-manifold obtained is not hyperbolic.
By Thurston Hyperbolic Dehn Surgery Theorem, it follows that most of the surgeries on a hyperbolic link  produce hyperbolic manifolds. We just need to exclude a finite number of slopes in each component. To determine which closed pure $3$-braids in $S^3$ are hyperbolic, we use a result of Birman and Menasco \cite{BM94}. The first author has shown examples
of hyperbolic closed pure $3$-braids which do have a nontrivial surgery producing $S^3$ (\cite{AS04}), some of these are small closed pure 3-braids.

The content of this paper is organized as follows: In Section \ref{hexatangle} we present for the convenience of the reader a short introduction to the hexatangle. In Section \ref{artin}  some examples taken from the tables of the Artin $3$-presentations of the trivial group,  obtained by considering  the symmetries of the hexatangle are presented, and in Section \ref{hiperbolicas} we determine all the hyperbolic, closed pure $3$-braids.

\section{The Hexatangle} \label{hexatangle}

We provide now the basic definitions, notation, and conventions for this paper, all of which are compatible with the ones in \cite{AE09}. In fact, this is a short introduction obtained from \cite{AE09}.

A tangle is a pair $(B,A)$, where $B$ is $S^3$ with the interior of a finite number of disjoint 3-balls removed, and  $A$ is the disjoint union of properly embedded arcs in $B$, with $\partial A \neq \emptyset$ such that $A$ intersects each component of $\partial B$ in four points. A marking of a tangle is an identification, of the points of intersection of $A$ in each component of $\partial B$, with exactly one of $\{ NE, NW, SW, SE\}$. A tangle $(B,A)$ with a marking is called a marked tangle.

We say that two marked tangles $(B,A_1)$ and $(B, A_2)$ are \textit{equivalent} if there is a homeomorphism of $B$, that fixes $\partial B$, and that takes $A_1$ to $A_2$ while preserving the marking.

A \textit{trivial $n$-tangle} is a tangle that is homeomorphic to the pair $(D^2, \{x_1, \ldots x_n \}) \times I$, where $x_i$ are points in the interior of $D^2$.

%\begin{center}
%\begin{figure}[h]
%\includegraphics[height=3.3cm]{marked2-tangle.png}
%\caption{Convention for marked tangles.}
%\label{marked2-tangle}
%\end{figure}
%\end{center}
A \textit{rational tangle} is a marked 2-tangle that is homeomorphic (as an unmarked tangle) to the trivial 2-tangle. Rational tangles can be parametrized by $\Q \cup \{ 1/0 \}$.  We denote the rational tangle corresponding to
$p/q\in Q\cup \{1/0\}$ by $R(p/q)$, the convention we are using is that the representation of $p/q$ as continued fraction:
$$
\frac{p}{q} = a_n+ \frac{1}{ a_{n-1}+ \frac{1} {... + \frac{1}{a_1} } } 
$$
corresponds to one of the tangles in Figure \ref{tangle-convention}(a) or (b), according to the parity of $n$ (see \textit{e.g.} \cite{Ka96}).
\begin{center}
\begin{figure}[ht]
\includegraphics[height=5.5cm]{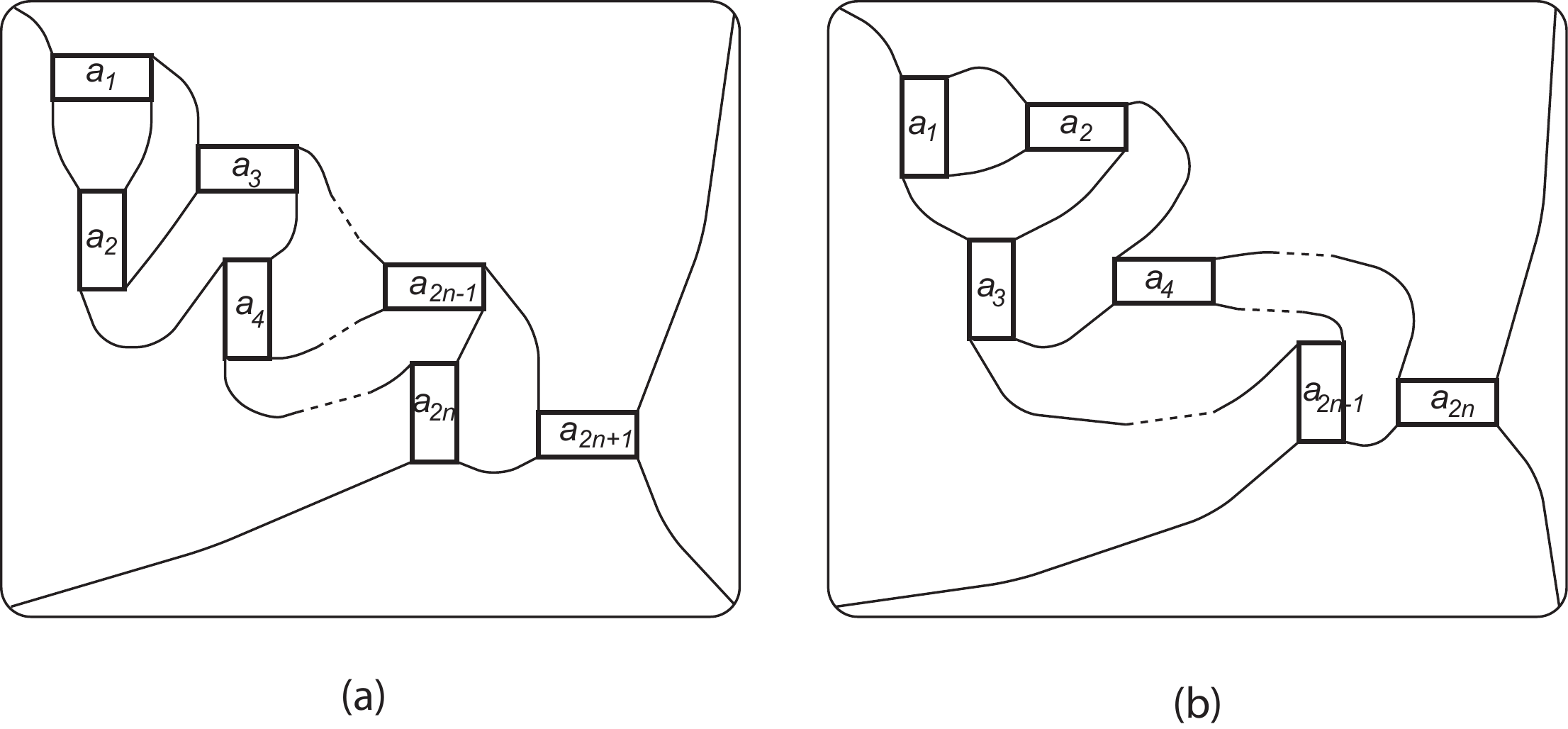}
\caption{Conventions for representing $p/q$}
\label{tangle-convention}
\end{figure}
\end{center}

The hexatangle is a marked tangle that produces a family of links obtained by filling the six 3-balls with integral tangles as in Figure \ref{hexat}.  Strictly speaking, a version of the hexatangle was first visualized in  \cite{Co69} where it was called basic polyhedra $6^*$ and $6^{**}$,  this was presented with a different marking. Although, Conway was more interested in these tangles to classify knots and links and to determine some relationships between the corresponding Conway polynomials.

\begin{center}
\begin{figure}[ht]
\includegraphics[height=6cm]{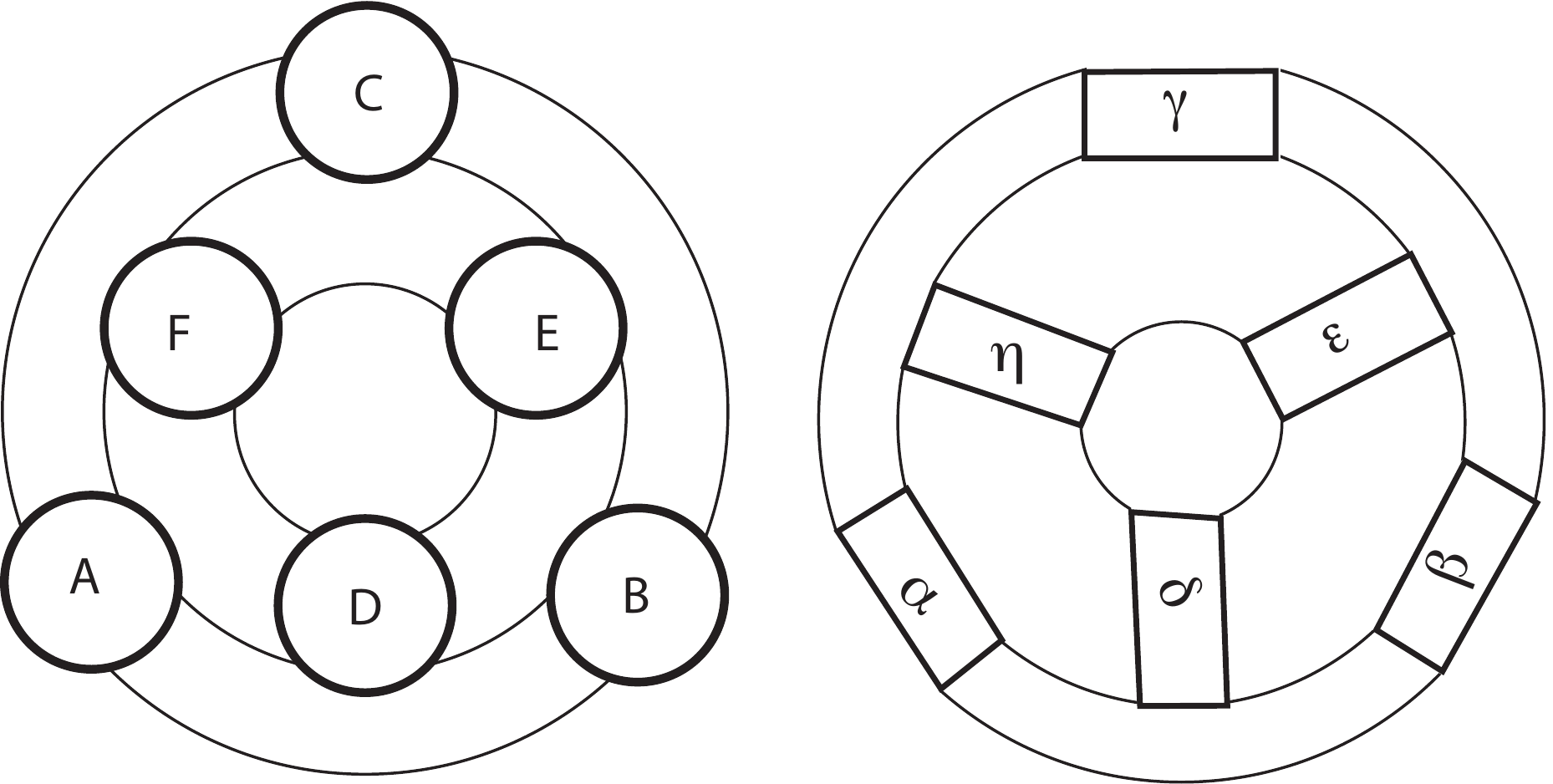}
\caption{The hexatangle}
\label{hexat}
\end{figure}
\end{center}

In \cite{AE09} the hexatangle is introduced as such, with the aim of giving information about integral Dehn surgery on the closure of certain  pure 3-braids $\widehat{\beta}$. %There it is determined exactly when an integral surgery in such $\widehat{\beta}$ produces the 3-sphere. 

Consider now the family of closed pure 3-braids of the form
$\hat \beta = \widehat{\sigma_1^{2e_1}\sigma_2^{2f_1}({\sigma_2\sigma_1\sigma_2})^{2e}}$, (see Figure \ref{fig4}).

\begin{center}
\begin{figure}[ht]
\includegraphics[height=6cm]{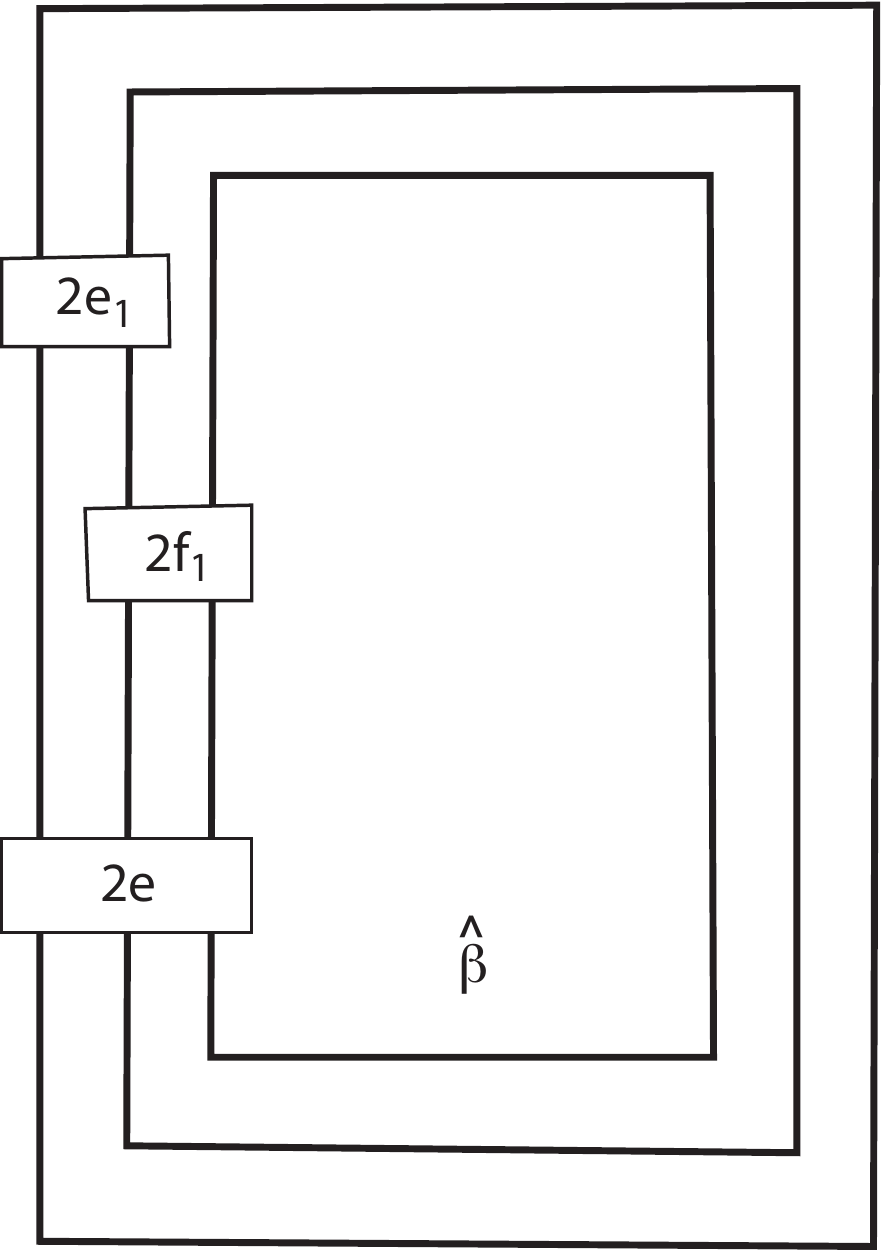}
\caption{Small  closed pure $3$-braids}
\label{fig4}
\end{figure}
\end{center}

\begin{center}
\begin{figure}[ht]
\includegraphics[height=6cm]{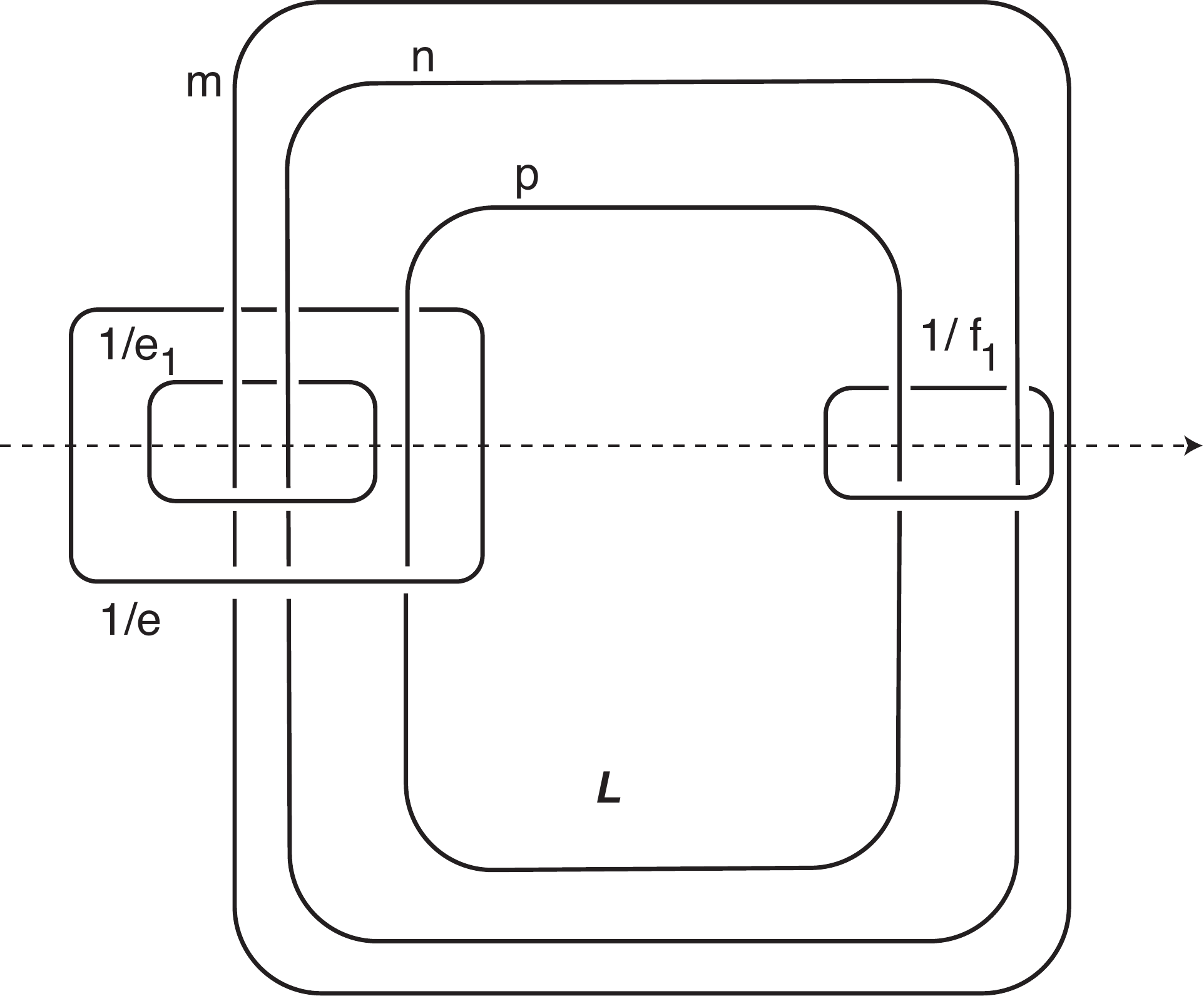}
\caption{Punctured $S^3$}
\label{fig:puncturedS3}
\end{figure}
\end{center}

The link $\hat {\beta}$ can be obtained by Dehn surgery over the link $\mathcal{L}$ of Figure \ref{fig:puncturedS3}, which is strongly invertible. An involution axis is shown in the same figure. The quotient of the exterior of $\hat{\beta}$ under this involution will be a punctured $S^3$, together with arcs, which arise as the image of the involution axis. The hexatangle is a class of framed links obtained by rearranging such punctured $S^3$ and making conventions of marked tangles. In \cite{AE09} the hexatangle is introduced, and it is determined precisely when an integral surgery in such links produces the 3-sphere.

\begin{center}
\begin{figure}[ht]
\includegraphics[height=6cm]{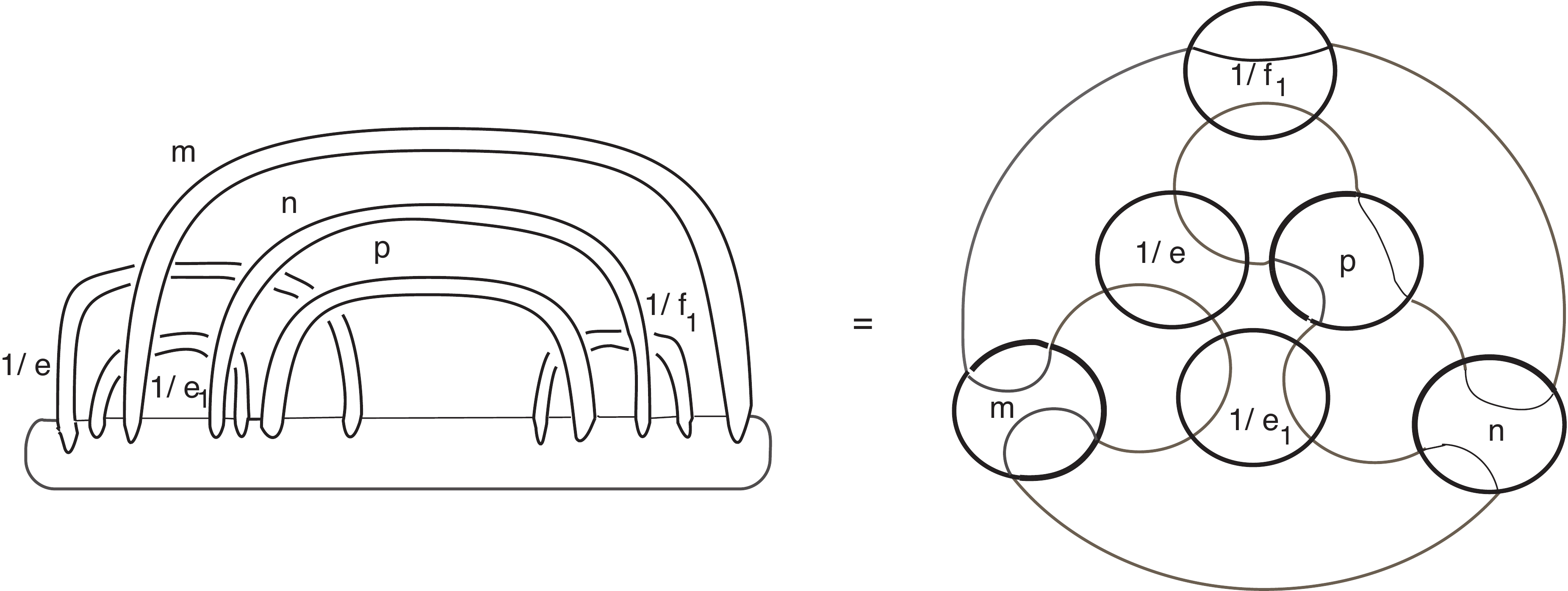}
\caption{The isotopy showing the origin of the hexatangle}
\label{fig5}
\end{figure}
\end{center}

The quotient of $\hat{\beta}$ under the involution is a tangle (see Figure \ref{fig5}), where its boundary components come from the tori boundary components of the exterior of $\hat{\beta}$, and the arcs are the image of the involution axis. We choose the marking given by the image of a framing on the components of $\hat{\beta}$ as shown in Figure \ref{fig5}. This is indicated 
in Figure \ref{hexat} by a rectangular box, where the short sides of the rectangle represent the axis $NW-SW$ and $NE-SE$, and the long sides represent the axis $NW-NE$ and $SW-SE$. In all of our pictures, the shape of the rectangle will always be clear.
We call this marked tangle the {\it Hexatangle}, and denote it by $\mathcal{H}$,
or $\mathcal {H}(*,*,*,*,*,*)$. The capital letters $A, B, C, D, E, F$ denote boundary components in the hexatangle, and $\alpha,\beta,\gamma,\delta,\epsilon,\eta$ denote fillings of the hexatangle 
with rational tangles. We refer to the sphere boundary components of $\mathcal {H}$, filled or unfilled, as  {\it boxes}. We say that two  boxes are  {\it adjacent} if there is an arc of $\mathcal {H}$ connecting them; otherwise, we will call them {\it opposite boxes}.
In the hexatangle, each box is opposite to exactly one box and adjacent to four boxes.
We consider $\alpha,\beta,\gamma,\delta,\epsilon,\eta$ as rational parameters that we use to filling the corresponding box with a rational tangle. Note that when we fill the boxes with integral tangles, we just replace each box with a sequence of horizontal crossings.

Note that filling one of the components $A$, $B$, $E$ with a rational tangle $\mathcal {R}(p/q)$ will correspond in the double branched cover to do $(-p/q)$-Dehn surgery on the corresponding component. On the other hand, filling with $\mathcal {R}(p/q)$ in one of the components $C$, $D$, $F$ corresponds to doing $q/p$-Dehn surgery in the corresponding component  in the double branched cover (see \cite{M75}), this because of our rational tangles convention (see \cite{Co69}). So, we can consider integral fillings
in all boundary components of the hexatangle and forget the correspondence  with the components of $\hat{\beta}$.
The hexatangle has many symmetries. Note that the hexatangle can be embedded in a tetrahedron so that each box is in correspondence with an edge of the tetrahedron; thus, each symmetry of the tetrahedron will give a symmetry of the hexatangle preserving framings. 
In \cite{AE09}, it is given a list of fillings on the hexatangle that produces the trivial knot up to symmetries, the list is complete up to the symmetries given by the tetrahedron and mirror images. In the present paper we are listing each symmetry of the hexatangle  and some representative Artin $3$-presentations of the trivial group that arise after each Dehn filling (we are not listing all the Artin $3$-presentations obtained, because of lack of space).

In \cite{AE09}, the main theorem (\ref{thm:main}) states the following:

\begin{thm}
\label{thm:main}
Suppose an integral  filling of the hexatangle produces the trivial knot,
then the parameters are exactly as shown in Tables \ref{tab:table1}, \ref{tab:table2} and \ref{tab:table3}, up to symmetries.
\end{thm}

%The tables mentioned in the  theorem are listed in Section 4.

\begin{table}[!ht]
    \tiny
    \centering
    \begin{tabular}{|c|c|c|c|c|c|c|}
    \hline
     & $\eta$ & $\beta$ & $\alpha$ & $\delta$ & $\epsilon$ & $\gamma$  \\
    \hline
     1& 0 & $\pm 1$ & $\pm 1$ & 0 & 0 & $\pm 1$  \\
     \hline
     2& 0 &0 & $\pm 1$ & 0 & $\pm 1$ & $\pm 1$ \\
    \hline
     & 0 & 1 & $\pm 1$ & 0 & -3 & -2  \\
    \hline
     & 0 &1    & $\pm 1$ & 0   & -2  & -3  \\
    \hline
    & 0 & 1  & $\pm 1$ & 0   & -1  & $\gamma$     \\
    \hline
    & 0 & 1   & $\pm 1$ & 0   & $\epsilon$  & -1     \\
    \hline
    & 0 & -2   & $\pm 1$ & 0   &  1 & -3  \\
    \hline
    & 0 & -3  & $\pm 1$ & 0   &  1 &  -2    \\
    \hline
    & 0 & $\beta$  & $\pm 1$ & 0   &  1 &  -1    \\
    \hline
    10 & 0 & -1  &$\pm 1$ & 0   & 1  & $\gamma$     \\
    \hline
    & 0 & 0 & 1 & 1 & -1 & -2  \\
    \hline
     & 0 &0    & 1 & 1   & -2  & -1  \\
    \hline
    & 0 & 0  & 1 & -1   & $\pm 1-\gamma$  & $\gamma$     \\
    \hline
    14& 0 & 0   &  -1 &  1   & $\pm 1 -\gamma$  & $\gamma$     \\
    \hline
    
    \end{tabular}
    \quad
     \begin{tabular}{|c|c|c|c|c|c|c|}
     
     \hline
     & $\eta$ & $\beta$ & $\alpha$ & $\delta$ & $\epsilon$ & $\gamma$  \\
    \hline
    15& 0 & 0   & -1 & -1   &  1 & 2  \\
    \hline
    & 0 & 0  & -1 & -1   &  2 &  1    \\
    \hline
    & 0 & 0  & 1 & -1   &  1 &  -2    \\
    \hline
     & 0 & 0  &1 & -2   & 1  & -1     \\
    \hline
     & 0 & 0   &  1  & $\pm 1 - \gamma$  & -1  & $\gamma$    \\
    \hline
    20 & 0 & 0   &  -1  & $\pm 1 - \gamma$  &  1 &  $\gamma$   \\
    \hline
    & 0 & 0   &  -1  & 1  & -1  &   2  \\
    \hline
     & 0 & 0   &  -1  & 2  & -1  &   1  \\
    \hline
    & 0 & 0   &  1  &  -2 & -1  &  1   \\
    \hline
    & 0 & 0   &  1  & -1  & -2  &   1  \\
    \hline
    & 0 & 0   &  1  & $\delta$  & $\pm 1 -\delta$  &  -1   \\
    \hline
    & 0 & 0   &  -1  &  $\delta$ & $\pm 1 -\delta$  &  1   \\
    \hline
    & 0 & 0   &  -1  & 2  & 1  &   -1  \\
    \hline
    28& 0 & 0   &  -1  &  1 & 2  &  -1   \\
    \hline
    
    \end{tabular}
    \caption{Integer fillings of the hexatangle that produce the trivial knot, part 1 of 3}
    \label{tab:table1}
\end{table}

\begin{table}[!ht]
    \tiny   
    \centering
     \begin{tabular}{|c|c|c|c|c|c|c|}
    
    \hline
     & $\eta$ & $\alpha$ & $\beta$ & $\gamma$ & $\delta$ & $\epsilon$  \\
    \hline 
    1&0  & 1   &  -1  &$\gamma $  &-1  & $\pm 1-\gamma  $ \\
    \hline
     & 0 &  1  & -1   &-2  &-2   &  -3  \\
    \hline 
    &0  &  1  & -1   &-3  & -2  & -2   \\
    \hline 
    & 0 &  1  &  -1  &-1  & -3  & -2   \\
    \hline 
    &0  & 1   & -1   &-2  &-3   &-1    \\
    \hline 
    &0  & 1   & -1   & -1 & -4  &-1    \\
    \hline 
    &0  & 1   & -1   &$\gamma $ & -2  & -1   \\
    \hline 
    & 0 & 1   & -1   & -1 & -2  &$ \epsilon $  \\
    \hline 
    & 0 & 1   & -1   &1  &1   & 2   \\
    \hline 
    10&0  &1    & -1   &2  &1   & 1   \\
    \hline 
    &0  & 1   &-1    &1  & 2  & 1   \\
    \hline 
    & 0 & 1   &$ \beta $  &$\gamma $ &-1   &$\pm 1-\gamma$    \\
    \hline 
    & 0 &  1  &   -3 & 1 & -2  &    $\epsilon$ \\
    \hline 
    & 0 &   1 &  -3  & $\gamma$ &  -2 & 1   \\
    \hline 
    &  0& 1   &-3    &3  & -2  & 2   \\
    \hline 
    & 0 & 1   & -3   & 2 &-2   &3    \\
    \hline 
    &0  & 1   & -2   &2  &-3   &1    \\
    \hline 
    & 0 & 1   &  -2  & 1 &-3  & 2   \\
    \hline 
     & 0 &  1  &  -1  & 1 & 2  & 1   \\
    \hline 
    20& 0 & 1   &-1    & 2 & 1  & 1   \\
    \hline 
     &  0& 1   & -2   &$\gamma$  &$-3-\gamma $  &1    \\
    \hline 
    & 0 &  1  & -3   & $\gamma $& -2  & 1   \\
    \hline 
      & 0   & 1   &$\beta $ &-1   &-2 &1    \\
    \hline 
    &0  &  1  &$\beta $   &-2  &-1   &1    \\
    \hline 
    &0  & 1   &-3    &-2  & $\delta $ &1    \\
    \hline 
    & 0 & 1   &  -2  &$ \gamma$ &$-\gamma -1 $  &  1  \\
    \hline 
    &0  & 1   &-3    &-3  & -4  &  1  \\
    \hline 
    &0  & 1   &-3    &-4  & -3  & 1   \\
    \hline 
    & 0 &  1  &  -4  &-2  &-3   &1    \\
    \hline 
    
    30 &0  & 1   &-4    &-3  & -2  &1    \\
    \hline 
     &  0&  1  & -5   &-2  &  -2 & 1   \\
    \hline 
    32& 0 &  1  &$\beta $   &1  &-2   &-1    \\
    \hline 
    \end{tabular}
    \quad
     \begin{tabular}{|c|c|c|c|c|c|c|}
    \hline
     & $\eta$ & $\alpha$ & $\beta$ & $\gamma$ & $\delta$ & $\epsilon$  \\
    \hline
    33& 0 &  1  & -1   &$\gamma $ & -2  & -1   \\
    \hline 
    & 0 &1    & 1   &3  &2   & -1   \\
    \hline 
    & 0 & 1   & 2   &2  & 1  & -1   \\
    \hline 
    & 0 & 1   & 1   &4  & 1  & -1   \\
    \hline 
    & 0 & 1   & $ \beta$  & 2 & -1  & -1   \\
    \hline 
    &0  & 1   & 1   & 2 &$\delta $  & -1   \\
    \hline 
     &0  &1    & -1   &-1  & -4  & -1   \\
    \hline 
    40 & 0 &   1 &-2    & -1 &-2   & -1   \\
    \hline 
     & 0 &  1  &  -1  &-2  & -3  &-1    \\
    
    \hline 
     & 0 &  1  &  -1  & 1 & 2  & 1   \\
    \hline 
    & 0 & 1   &-1    & 1 & 1  & 2   \\
    \hline 
     &  0& 1   & -2   &1  &$-3-\epsilon $  &$\epsilon$    \\
    \hline 
    & 0 &  1  & -3   & 1 & -2  & $\epsilon$   \\
    \hline 
      & 0   & 1   &$\beta $ &1   &-2 &-1    \\
    \hline 
    &0  &  1  &$\beta $   & 1 &-1   &-2    \\
    \hline 
    &0  & 1   &-3    &1  & $\delta $ &-2    \\
    \hline 
    & 0 & 1   &  -2  &1 &$-\epsilon -1 $  &  $\epsilon$  \\
    \hline 
    50&0  & 1   &-3    &1  & -4  &  -3  \\
    \hline 
    &0  & 1   &-3    &1  & -3  & -4   \\
    \hline 
    & 0 &  1  &  -4  &1  &-3   &-2    \\
    \hline 
    
     &0  & 1   &-4    & 1 & -2  &-3    \\
    \hline 
     &  0&  1  & -5   &1  &  -2 & -2   \\
    \hline 
    & 0 &  1  &$\beta $   &-1  &-2   &1    \\
    \hline 
    & 0 &  1  & -1   &-1 & -2  & $\epsilon$   \\
    \hline 
    & 0 &1    & 1   &-1  &2   & 3   \\
    \hline 
    & 0 & 1   & 2   &-1  & 1  & 2   \\
    \hline 
    & 0 & 1   & 1   &-1  & 1  & 4   \\
    \hline 
    60& 0 & 1   & $ \beta$  & -1 & -1  & 2   \\
    \hline 
    &0  & 1   & 1   & -1 &$\delta $  & 2   \\
    \hline 
     &0  &1    & -1   &-1  & -4  & -1   \\
    \hline 
     & 0 &   1 &-2    & -1 &-2   & -1   \\
    \hline 
    64 & 0 &  1  &  -1  &-1  & -3  &-2    \\
    \hline
    
    \end{tabular}
    \caption{Integer fillings of the hexatangle that produce the trivial knot, part 2 of 3}
    \label{tab:table2}
\end{table}

\begin{table}[!ht]
    \tiny
    \centering
     \begin{tabular}{|c|c|c|c|c|c|c|}
    
    \hline
     & $\delta$ & $\eta$ & $\alpha$ & $\beta$ & $\gamma$ & $\epsilon$  \\
    \hline
    1& -1 & 1   &  1  &$ \beta$ &-2   & 2   \\
    \hline 
    &-1  & 1   &1    & -1 &-4   &2    \\
    \hline 
    & -1 & 1   & 1   &2  & -1  & 2   \\
    \hline 
    &-1  & 1   & 1   & -2 &-3   & 2   \\
    \hline 
    & -1 & 1   &-1    &2  &$ \gamma$  & -2   \\
    \hline 
    & -1 &   1 &-1    & 4 &1   & -2   \\
    \hline 
    &-1  & 1   & -1   &1  & -2  &  -2  \\
    \hline 
    & -1 & 1   &  -1  &  3&2   &-2    \\
    \hline 
    &-1  & 1   & 2   & 2 & -1  &1    \\
    \hline 
    10 & -1 &  1  &2    & -1 & -2  & 1   \\
    \hline 
    &-1  & 1   & 2   & 1 &-2   &1    \\
    \hline 
    &-1  &1    &-2    &1  &-2   &  -1  \\
    \hline 
    & -1 &1    &-2    &2  & 1  & -1   \\
    \hline 
    & -1 & 1   & -2   & 2 & -1  &-1    \\
    \hline 
    & -1 & 1   &  $\alpha $ &$-\alpha +1 $ &-2   &1    \\
    \hline 
    &-1  &1    &  2  & 3 &-2   &  3  \\
    \hline 
    &-1  & 1   & 2   & 5 & -2  & 2	   \\
    \hline 
    &-1  & 1   &3    & 4 & -2  & 2   \\
    \hline 
    & -1 & 1   &  1  & 3 &-2   & 4   \\
    \hline 
    20 & -1 &  1  & 1   &4  &-2   & 3   \\
    \hline
    
    \end{tabular}
    \quad
    \begin{tabular}{|c|c|c|c|c|c|c|}
    
    \hline
     & $\delta$ & $\eta$ & $\alpha$ & $\beta$ & $\gamma$ & $\epsilon$  \\
    \hline 
    21& -1 &  1  &  -1  &2  &-2   & $\epsilon $  \\
    \hline 
    &-1  & 1   & 1   &$\beta$  & -2  & 2   \\
    \hline
    &-1 & 1   &$\alpha$   &3  &-2   & 2   \\
    \hline
    & -1 & 1   & $\alpha $  &$3-\alpha $ &-2   & 1   \\
    \hline
    & -1 & 1   &-1   & 1 &-2   & -2   \\
    \hline
    &-1  & 1   &  -2  &1  & -2  &  -1  \\
    \hline
    & -1 & 1   &  1  & 2 &-2   & 3   \\
    \hline
    &-1  & 1   &  1  & 2 &  -2 & 4   \\
    \hline
     &-1  & 1   & 1   &2  & -2  &$ \epsilon $  \\
    \hline
    30& -1 & 1   &$ \alpha$   &2  &$-1-\alpha $  & -1   \\
    \hline
    &-1  &1    & -2   &2  &  -3 & -3   \\
    \hline
    &-1  & 1   & -2   &2  &-5   &-2    \\
    \hline
    & -1 &  1  & -3   &2 & -4  & -2   \\
    \hline
    &-1  & 1   &   -1 &2  &-3   & -4   \\
    \hline
    &-1  & 1   &  -1  &2  &-4   & -3   \\
    \hline
    &-1  &1    &$\alpha $   &1  & -3  &-2    \\
    \hline
    &-1  & 1   &   -1 &2  &$\gamma $  &-2    \\
    \hline
    & -1 &  1  & $ \alpha $ &2  &$-3-\alpha$   &-1    \\
    \hline
    &-1  & 1   & 1   &2  & -1  & 2   \\
    \hline
    40& -1 & 1   & 2   & 2 &-1   &  1  \\
    \hline
    
     \end{tabular}
    
    \caption{Integer fillings of the hexatangle that produce the trivial knot, part 3 of 3}
    \label{tab:table3}
\end{table}

\begin{lemma}
The double branched cover of $\mathcal{H}(\alpha,\beta,\gamma,\delta,\epsilon,\eta)$ is obtained by $(-\alpha-\delta-\eta,-\beta-\delta-\gamma-\eta,-\epsilon-\gamma-\eta)$-Dehn surgery on the closed pure 3-braid $(\sigma_1)^{-2\delta}(\sigma_2)^{-2\gamma}(\sigma_2\sigma_1\sigma_2)^{-2\eta}$.
\end{lemma}

\begin{proof}
By observing  the hexatangle (see Figure \ref{hexat}), we have the following correspondences:
 
$$\alpha \rightarrow -m,\hskip 6pt \beta \rightarrow -n,\hskip 6pt \gamma \rightarrow f_1, \hskip 6pt \eta \rightarrow e, \hskip 6pt\delta\rightarrow e_1, \hskip 6pt \epsilon \rightarrow -p$$
 
This means that the double branched cover of $\mathcal{H}=(\alpha,\beta,\gamma,\delta,\epsilon,\eta)$ is obtained by Dehn surgery on $\mathcal{L}(-m,-n,1/\gamma,1/\delta,-p,1/\eta)$. After performing surgery $(1/\gamma,1/\eta,1/\delta)$ on $\mathcal{L}$, we get the closed pure 3-braid $(\sigma_1)^{-2\delta}(\sigma_2)^{-2\gamma}(\sigma_2\sigma_1\sigma_2)^{-2\eta}$, and the surgery coefficients change to $(-\alpha-\delta-\eta,-\beta-\delta-\gamma-\eta,-\epsilon-\gamma-\eta)$. 
\end{proof}

\section{Artin $n$-Presentations of the Trivial Group}\label{artin}

When we consider a group $G = <x_1 x_2,\cdots , x_n : r_1, r_2, \cdots , r_n>$, we say that this is an F-Artinian $n$-presentation if the following equation is satisfied in the free group $F(x_1, x_2, \cdots , x_n)$:
$$\prod r_ix_i r_i^{-1} = \prod x_i$$
 F-Artinian $n$-presentations were defined by Gonz\'alez-Acu\~na in the 70's, see \cite{GA75}.
We say that a W-Artinian $n$-presentation is taken if the following equation is satisfied in the free group $F(x_1, x_2, \cdots , x_n)$:
$$\prod r_i^{-1}x_i r_i = \prod x_i$$
In this paper, we are considering W-Artinian $n$-presentations.

These kinds of presentations appear when we consider the fundamental group of a closed, connected, and orientable $3$-manifold, given by an open-book decomposition with planar pages.

From this open-book decomposition of the $3$-manifold, its fundamental group can be read. Just read each relation described by  certain disjoint simple curves $r_i$ for $i=1,..., n$. As an example, we give the figure \ref{fig:fig6}.

\begin{center}
\begin{figure}[ht]
\includegraphics[height=8cm]{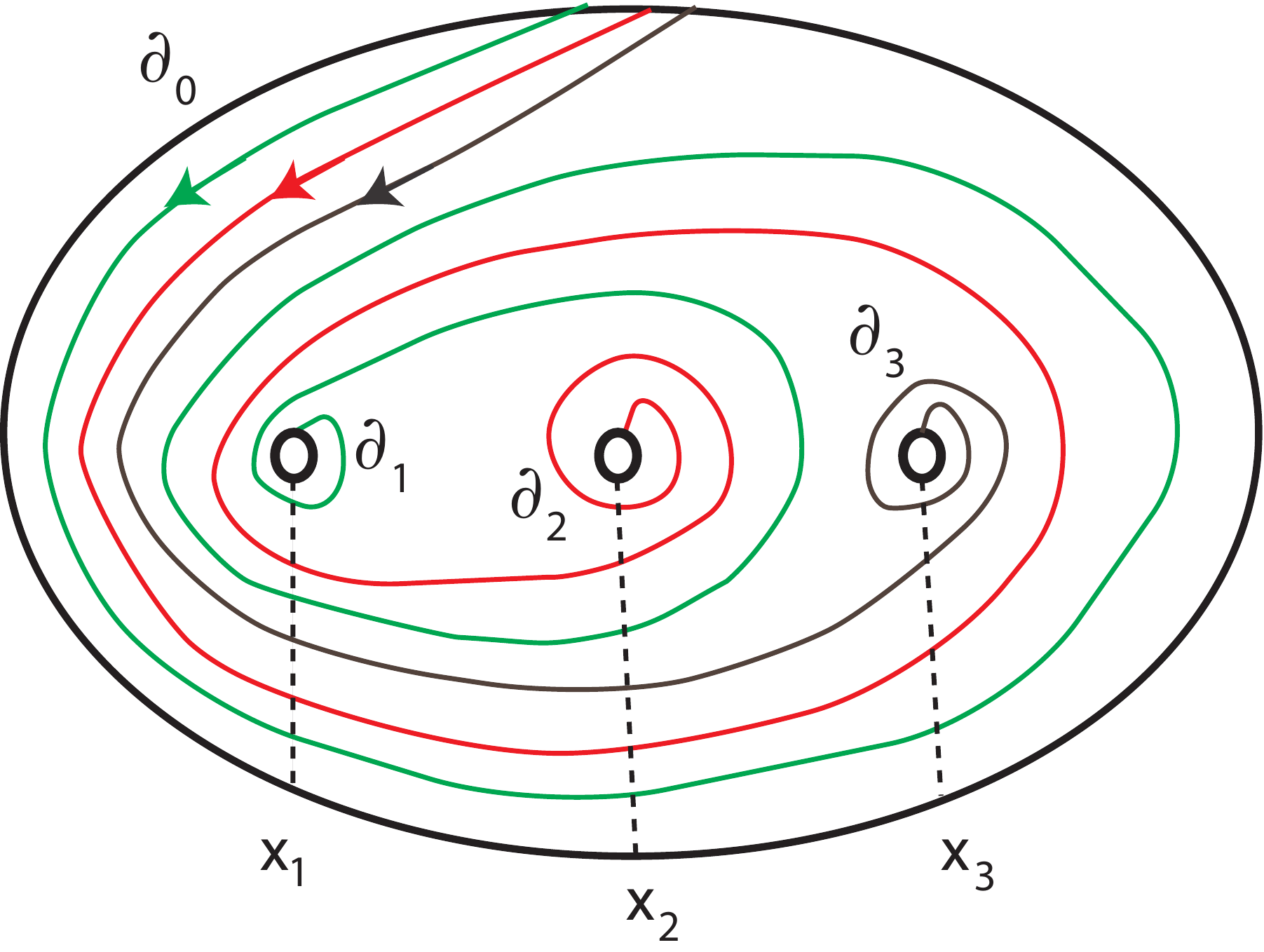}
\caption{The open book decomposition with planar page of a 3-manifold}
\label{fig:fig6}
\end{figure}
\end{center}
The relations are the following:
$$r_1 = x_1x_2x_3x_1x_2x_1$$
$$r_2= x_1x_2x_3x_1x_2^2$$ and $$r_3= x_1x_2x_3^2$$

Since all closed, connected, and orientable $3$-manifold can be described in this way ( see \cite{GA75}),  there is no loss of generality in  considering these kinds  of presentations when we study fundamental groups of such  $3$-manifolds.

Now, we consider the symmetries of the hexatangle. Since the hexatangle can be seen as a tetrahedron, there are $24$ symmetries. The symmetries are as shown in Table \ref{tab:table4}, where for example, symmetry $3$ is saying that $\alpha \rightarrow \gamma, \beta \rightarrow
\alpha, \gamma \rightarrow \beta, \delta \rightarrow \eta, \eta \rightarrow \epsilon$ and $\epsilon \rightarrow \delta$. It is meaning that $\alpha$ is taking the value of $\gamma$, $\beta$ is taking the value of $\alpha$, and so on in Table \ref{tab:table4}. 
\begin{table}[!ht]
    \centering
    \begin{displaymath}
    \begin{array}{|r|c|c|c|c|c|c|}
    \hline
    \mathrm{Symmetry }& \alpha & \beta &\gamma &\delta& \eta &\epsilon  \\
     & & & & & & \\
    \hline 
    \hline 
     1& \alpha & \beta & \gamma & \delta & \eta & \epsilon  \\
    \hline 
    2& \beta &\gamma  &\alpha  &\epsilon  &\delta  &\eta  \\
    \hline 
     3& \gamma &\alpha  &\beta  &\eta  &\epsilon  &\delta  \\
    \hline 
    4&\gamma &\epsilon  &\eta  &\beta  & \alpha &\delta  \\
    \hline 
     5&\eta  & \delta & \alpha & \epsilon &\gamma  &\beta  \\
    \hline 
     6&\beta  &\delta  &\epsilon  &\alpha  &\gamma  &\eta  \\
    \hline 
     7&\delta  & \alpha &\eta  &\beta  &\epsilon  &\gamma  \\
    \hline 
     8&\delta  &\epsilon  &\beta  & \eta &\alpha  &\gamma  \\
    \hline 
     9&\eta  & \gamma &\epsilon  &\alpha  & \delta & \beta \\
    \hline 
     10&\alpha  &\eta  &\delta  &\gamma  & \beta & \epsilon\\
    \hline 
     11& \epsilon &\beta  &\delta &\gamma  & \eta &\alpha  \\
    \hline 
     12&\epsilon  & \eta &\gamma  &\delta  &\beta  & \alpha \\
    \hline 
     13&\alpha  &\gamma  &\beta  &\eta  &\delta  &\epsilon  \\
    \hline 
     14&\beta  &\alpha  &\gamma  &\delta  &\epsilon  &\eta  \\
    \hline 
     15&\gamma  &\beta  & \alpha &\epsilon  &\eta  & \delta \\
    \hline 
     16&\gamma  &\eta  & \epsilon &\alpha  &\beta & \delta \\
    \hline 
     17& \eta &\alpha  &\delta  & \gamma & \epsilon & \beta \\
    \hline 
     18&\beta   &\epsilon  & \delta &\gamma  & \alpha &\eta  \\
    \hline 
     19&\delta  &\eta  &\alpha  & \epsilon & \beta &\gamma  \\
    \hline 
     20&\delta  & \beta &\epsilon  &\alpha  &\eta  &\gamma  \\
    \hline 
     21&\eta  &\epsilon  &\gamma  &\delta &\alpha  &\beta  \\
    \hline 
     22&\alpha  & \delta &\eta  &\beta  &\gamma  & \epsilon \\
    \hline 
     23&\epsilon  &\delta  &\beta  &\eta  &\gamma  &\alpha  \\
    \hline 
     24&\epsilon &\gamma  &\eta  & \beta &\delta  &\alpha  \\
    \hline
      \end{array}
     \end{displaymath} 
    \caption{Symmetries of the hexatangle.}
    \label{tab:table4}
\end{table}
     
In \cite{AS12}, the following theorem is proved:
\begin{thm}
Let $ \hat{\beta} \in S^3$ be the closed pure 3-braid of $\beta= \Delta^{2e} {(\sigma_1 ^2)}^{e_1}{(\sigma_2 ^2)}^{f_1}$ where $\Delta = \sigma_1\sigma_2\sigma_1 = \sigma_2\sigma_1\sigma_2$ with $ e, e_1,$ $f_1$ $\in \mathbb{Z}$. Then the
3-manifold obtained by Dehn surgery on $\hat{\beta}$ with an integral framing $(m, n, p)$ has the Artin $3$-presentation given by the generators $x_1$, $x_2$, $x_3$ and relations:
\begin{align*}
r_1 &= x_1^{m-e-e_1} (x_1(x_2x_3)^{-f_1} x_2 (x_2 x_3)^{f_1})^{e_1} (x_1 x_2 x_3) ^e\\
r_2 &= x_2^{n-e-e_1-f_1} (x_2 x_3)^{f_1}(x_1 (x_2 x_3)^{-f_1} x_2(x_2 x_3)^{f_1})^{e_1}(x_1 x_2 x_3)^ e\\
r_3 &= x_3^{p-e-f_1} (x_2 x_3)^{f_1}(x_1 x_2 x_3)^e 
\end{align*}
\end{thm}

In \cite{AS12} there is a mistake in the relations $r_1$ and $r_2$ in the part of the relations $(x_1(x_2x_3)^{f_1}x_2(x_2x_3)^{-f_1}$ where the signs of the exponent $f_1$ were interchanged. Here this mistake is corrected. From this theorem and from Lemma 2.2 it follows

\begin{cor} The Artin $3$- presentation associated to the double branched cover of $\mathcal{H}(\alpha,\beta,\gamma,\delta,\epsilon,\eta)$ is
%to the closed pure 3-braid
%(\sigma_1)^{-2\delta}(\sigma_2)^{-2\delta}(\sigma_2\sigma_1\sigma_2)^{-2\eta}$, with surgery coefficients  $(-\alpha-\delta-\eta,-\beta-\delta-\gamma-\eta,-\epsilon-\gamma-\eta)$ is given by:

 $r_1 = x_1^{-\alpha} (x_1(x_2x_3)^{\gamma } x_2 (x_2 x_3)^{-\gamma })^{-\delta } (x_1 x_2 x_3) ^{-\eta} $

 $r_2 = x_2^{-\beta} (x_2 x_3)^{-\gamma }(x_1 (x_2 x_3)^{\gamma } x_2(x_2 x_3)^{-\gamma })^{-\delta }(x_1 x_2 x_3) ^{-\eta}$

 $r_3 = x_3^{-\epsilon } (x_2 x_3)^{-\gamma }(x_1 x_2 x_3) ^{-\eta}$
  
\end{cor}
  
The W-Artinian $3$-presentations for some examples of the trivial group see Tables \ref{tab:examples1}, \ref{tab:examples2}, \ref{tab:examples3}, \ref{tab:examples4}, \ref{tab:examples5}, and \ref{tab:examples6}.
%% TODO: Insertar tablas aquí, como unas 4 páginas 
\begin{table}[!ht]
    \input{tables/table_1}

    \caption{Examples of Artin n-presentations of the trivial group coming from
    the first part of Table \ref{tab:table1}}
    \label{tab:examples1}
\end{table}  

\begin{table}[!ht]
    \input{tables/table_2}

    \caption{Examples of Artin n-presentations of the trivial group coming from
    the second part of Table \ref{tab:table1}}
    \label{tab:examples2}
\end{table}  

\begin{table}[!ht]
    \input{tables/table_3}
    \caption{Examples of Artin n-presentations of the trivial group coming from
    the first part of Table \ref{tab:table2}}
    \label{tab:examples3}
\end{table}  

\begin{table}[!ht]
    \input{tables/table_4}
    \caption{Examples of Artin n-presentations of the trivial group coming from
    the second part of Table \ref{tab:table2}}
    \label{tab:examples4}
\end{table}  

\begin{table}[!ht]
    \input{tables/table_5}
    \caption{Examples of Artin n-presentations of the trivial group coming from
    the first part of Table \ref{tab:table3}}
    \label{tab:examples5}
\end{table}  

\begin{table}[!ht]
    \input{tables/table_6}
    \caption{Examples of Artin n-presentations of the trivial group coming from
    the second part of Table \ref{tab:table3}}
    \label{tab:examples6}
\end{table}  

\section{Hyperbolic closed pure 3-braids} \label{hiperbolicas}

In this Section, we prove which closed pure $3$-braids are hyperbolic. To do this, we use a result of Birman and Menasco.

In \cite{AS04}, we proved that the group of  3-pure braids is the direct product of the groups $ Z$ and the free group in two generators. This does permit us to give a general diagram to represent a 3-pure braid, which is given by 
$$\beta =\prod _{i=1}^n \sigma_1^{2e_i}\sigma_2^{2f_i}(\sigma_1\sigma_2\sigma_1)^{2e}$$
where $e_i,f_i$ and $e$ are integers. We call to $\beta = \sigma _1 ^{2e_1} \sigma _2^{2f_1}(\sigma_1\sigma_2\sigma_1)^{2e}$ the small case.

Also, in \cite {AE09} some hyperbolic closed pure 3-braids are described, and after performing integral Dehn surgery $S^3$ is obtained. So, these are exceptional Dehn fillings coming from the hexatangle.
Here, we said which  closed pure 3-braids are hyperbolic, not only in the small case, but in the general case. The result of Birman and Menasco  is the following corollary. Here is a small introduction taken from \cite {BM94} to understand it.
 Let  $L$ be a closed 3-braid with axis A. A is unknotted, so $S^3 - A$ is fibred by open disks $H = \{H_{\theta} : \theta \in [0,2\pi] \}$. Suppose $T$ is not a peripheral torus. The torus $T$ of type 0 is described as follows (see Figure \ref{fig:birman1}).  The torus $T$ is the boundary of a (possible knotted) solid torus $V$ in $S^3$ whose core is a closed braid with axis $A$. The link $L$ is also a closed braid with  respect to $A$, part of $L$ inside $V$ and part (possibly empty) outside. The torus $T$ is transverse to all fiber $H_\theta$  in the fibration of $S^3 - A$, and intersects each fiber in a meridian disk of $V$.

\begin{center}
\begin{figure}[ht]
\includegraphics[width=6cm]{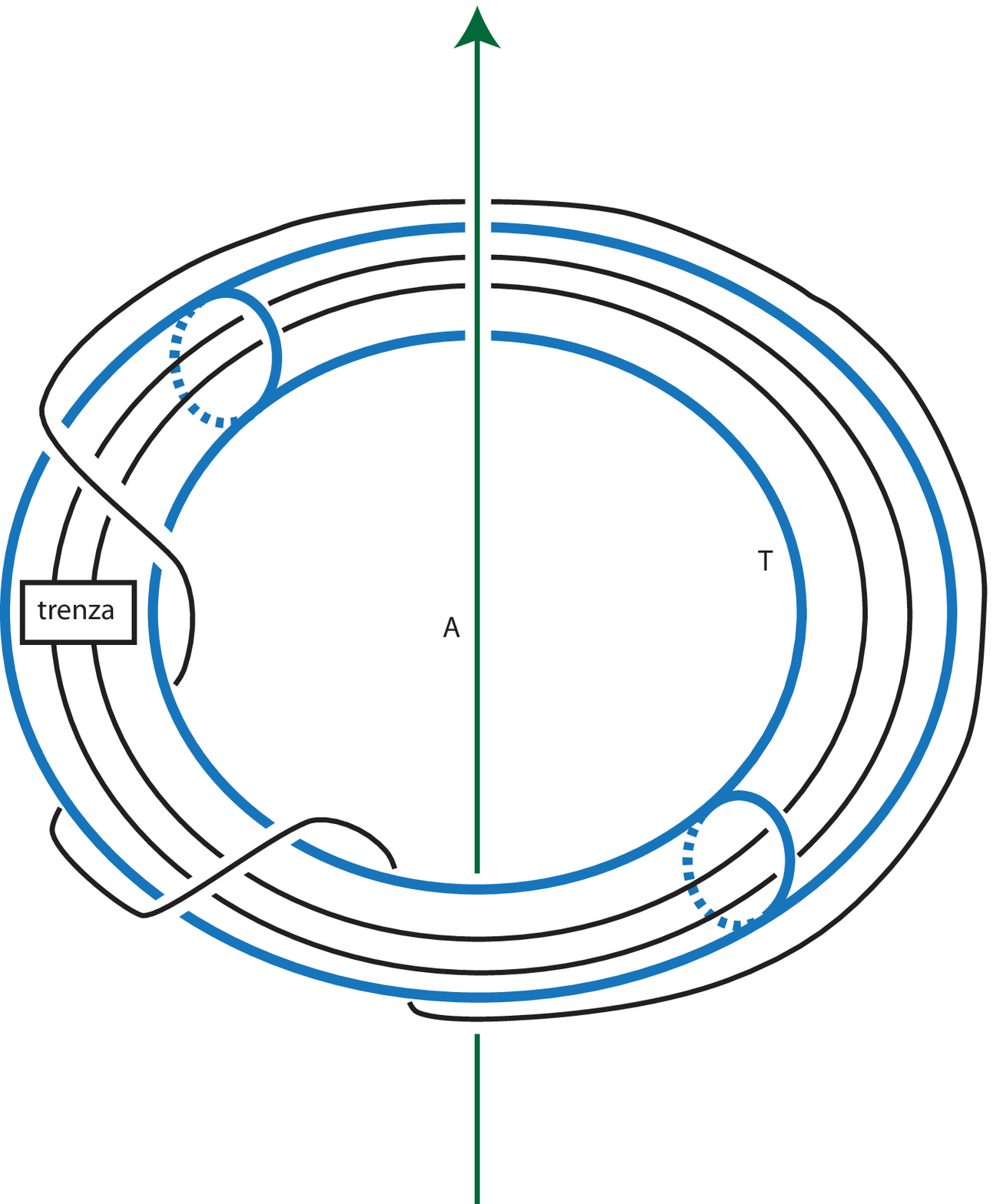} 
\caption{Torus of type 0. This figure was taken from \cite{BM94}}
\label{fig:birman1}
\end{figure}
\end{center}

 \begin{cor}
  Let ${\bf L}$ be a closed 3-braid or a prime link of braid index  3, and let ${\bf T}$  be an essential torus in  $S^3 -{\bf L}$. Then ${\bf T }$ has an embedding of type $0$ and ${\bf L}$ is conjugated to  $$\beta = (\sigma _2)^k(\sigma _1\sigma_2^2\sigma_1)^q$$

 \noindent  where $ |k|\geq2, |q|\geq1.$
\end{cor}

In this work, we determine which  closed pure 3-braids are equivalent to $ (\sigma _2)^k(\sigma _1\sigma_2^2\sigma_1)^q$. Observe that a condition for  $\beta$ to be pure is that $k$ is an even number.
 
We observe that although the result of Birman and Menasco is incomplete in the case of closed $n$-braids, in the case of closed $3$-braids is complete, we use it to determine which closed pure $3$-braids are hyperbolic.

In \cite{BM94} Corollary 1, it is shown  that if a closed 3-braid $\hat \beta$ admits an essential
torus, then  $ \beta $  is conjugate to $$(\sigma _2)^k( \sigma _1\sigma _2^2 \sigma _1)^q$$ where
$|k|\geq 2, |q|\geq 1$.
Here, we determine which small closed pure 3-braids are equivalent to $\hat \beta$.
In our case, $k$ is an even number.

Let $\rho : B_3 \rightarrow {\mathbb{Z}}_2 * {\mathbb{Z}}_3$ be such that
$$B_3 = < \sigma_1, \sigma_2 | \sigma_1\sigma_2\sigma _1 =\sigma_2\sigma_1\sigma_2>$$
and $${\mathbb{Z}}_2 * {\mathbb{Z}}_3 = <\Delta, y | \Delta^2 = 1, y^3 = 1>$$
where $$ \sigma_1  \mapsto y^2\Delta , \sigma_2  \mapsto \Delta y^2$$

Then it is known that $\rho$ is an epimorphism, it comes from the fact that $B_3$ is homeomorphic to the group of the trefoil knot, and the latter has a quotient to ${\mathbb{Z}}_2 * {\mathbb{Z}}_3$. Observe that we are abusing the notation, $\Delta$ represents here a generator of ${\mathbb{Z}}_2 * {\mathbb{Z}}_3$, but also represents the $\Delta = \sigma_1\sigma_2 \sigma_1 = \sigma_2\sigma_1 \sigma_2$ in the 3-braid group $B_3$. This abuse is justified by the relation $\rho(\Delta) = \Delta$. Notice also that $\rho(\sigma_1 \sigma_2) = y$.

%Observe that  $\sigma_1 = y^2\Delta^{-1} = y^{-1} \Delta$

We will use $\rho$ to determine when the words $$w_1 = \sigma_1^{2e_1}\sigma_2
^{2f_1}\Delta ^{2e}, \quad w_2 = \sigma_2^{2q}(\sigma_1\sigma_2^2\sigma_1)^{2k}$$
are conjugated.  

It follows that $\rho (w_1) = (y^2\Delta)^{2e_1}(\Delta y^2)^{2f_1}$
and since $(y\Delta)^{-1} = \Delta y^2$ we obtain that: $$\rho(w_1) = (y^2\Delta)^{2(e_1-f_1)}$$

Now,  $\rho (w_2) = 
(\Delta y^2)^{2q}(y^2 \Delta\Delta y^2\Delta y^2y^2\Delta)^{2k} = (\Delta y^2)^{2q}(y\Delta y\Delta)^{2k}$

$$= \Delta y^2 \Delta y^2\cdots \Delta y^2 y\Delta y\Delta \cdots y\Delta$$
$$= (\Delta y^2 )^{2q-4k} = (y\Delta)^{4k-2q}.$$

Now, we proceed by cases, using that $\rho (w_1)$ is conjugated to $\rho(w_2)$ if and only if their minimal cyclically reduced words are equals (see \cite[Theorem 1.4]{MKS}).

Case 1. $e_1 \geq 1, f_1\geq 1$:

$$(y^2 \Delta)^{2e_1} (\Delta y^2)^{2f_1} =$$
$$ = y^2 \Delta (y^2\Delta (y^2 \Delta)^{2e_1-2} y^2 \Delta \Delta (y\Delta^2 y^2)^{2f_1 - 2}\Delta y^2 =$$
$$= y^2 \Delta (y^2 \Delta )^{2e_1 - 2}y(\Delta y^2)^{2f_1 - 2} \Delta y^2 =$$
$$= y \Delta (y^2\Delta)^{2e_1 -2} y (\Delta y^2)^{2f_1 - 2} \Delta $$

but $\rho (w_2) = (y\Delta)^{2(k-q)}$, so in this case $e_1 = 1$ and $f_1 = 1 $ and $k-q =1$.

Case 2. $e_1 < 1, f_1 \geq 0$:

$$( y^2 \Delta )^{2e_1}( \Delta y^2)^{2f_1} = (\Delta y)^{-2e_1}(\Delta y^2)^{2f_1}$$
and since the minimal cyclically reduced word of $\rho (w_2)$ only has $y^2$ or only $y$, it follows that
$e_1 = 0$ or $f_1 = 0$.
If $e_1 = 0$, we have  $(\Delta y^2)^{2f_1}$ which is a power of $(\Delta y^2)$, and if $f_1 =0$,
$(\Delta y)^{-2e_1}$ is a power of 
$\Delta y^2$ if and only if is a negative power and $(\Delta y)^{-2e_1}$ is equivalent to 
$y\Delta y\Delta\cdots y\Delta$, then $e_1 = 0$ or $f_1 = 0$.

Case 3. $e_1\geq 0, f_1< 1$:

Remind that $y\Delta y^2 $  is equivalent to $\Delta $ so,
$$(y^2 \Delta)^{2e_1}(\Delta y^2)^{2f_1} = (y^2 \Delta)^{2e_1}(y\Delta)^{-2f_1}$$
is equivalent to a power of $\Delta y^2$ if $e_1 = 0 $ or $f_1 = 0$.

Case 4. $e_1 \leq -1, f_1\leq -1$:

In  this case,
$$(y^2 \Delta )^{2e_1} (\Delta y^2)^{2f_1} = (\Delta  y)^{-2e_1}(y\Delta )^{-2f_1}$$
$$= \Delta y (\Delta y)^{-2e_1 - 2}
\Delta y y \Delta (y \Delta )^{-2f_1 - 2} y \Delta =$$
$$ = y^2 (\Delta y )^{-2e_1 - 2} \Delta y^2 \Delta (y \Delta )^{-2f_1 - 2 }$$
Then $e_1 = -1$ and $f_1 = -1$.

So, we have 
\begin{thm} Let $\beta = \sigma_1^{2e_1}\sigma_2^{2f_1}\Delta ^{2e}$ be a pure $3$-braid.Then the closed pure $3$-braid $\hat\beta $ has an essential  torus if: 

i)$e_1 = 0$, or $f_1 = 0$ and $e$ any integer. Observe that if $e = 0$ then $\hat \beta$ is splittable.   

ii) $e_1 = f_1 = \pm 1$ \hskip 5pt and $e$ any integer.

iii) $e_1$, $ f_1 \in \mathbb{Z} - \{0\}$  
 and $ e = 0$. Observe that in this case we have a connected sum.
\end{thm}

We give some examples, see Figure \ref{fig:trenzanillo}.

\begin{center}
\begin{figure}[ht]
\includegraphics[width=10cm]{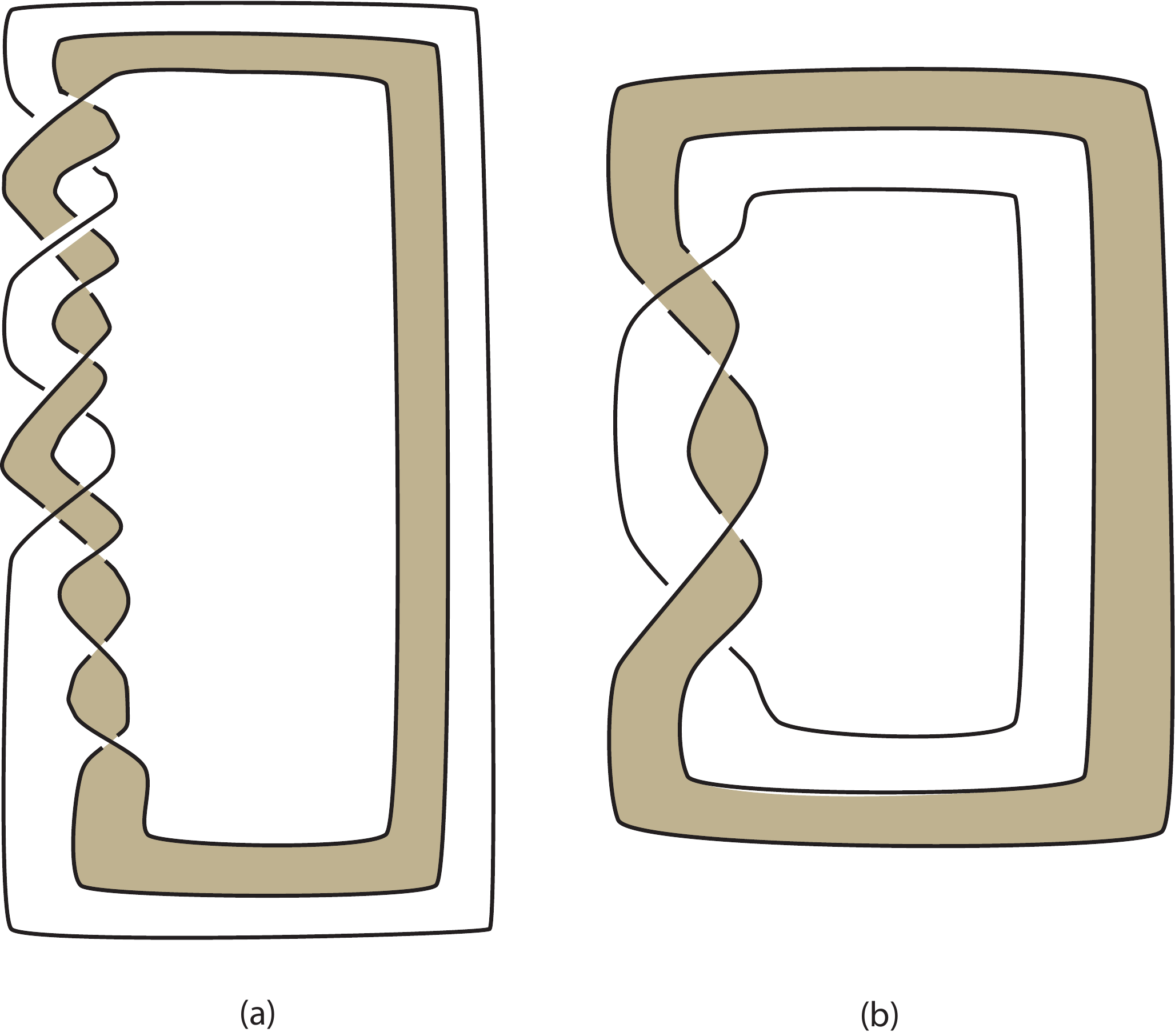} 
\caption{An annulus in the closed pure 3-braid (a) $(\sigma_2\sigma_1\sigma_2)^4\sigma_2^{-2}$;\hskip10pt(b) $ (\sigma_2\sigma_1\sigma_2)^2$.}
\label{fig:trenzanillo}
\end{figure}
\end{center}

To prove the next theorem we have the following.

\vskip 10pt

\begin{propst} Let $w = \prod _{i=1}^n (y^2\Delta)^{2e_i}(\Delta y^2)^{2f_i}$
be a word in $\mathbb{Z}_2\langle\Delta\rangle* \mathbb{Z}_3\langle y \rangle$ with $e_i,f_i\neq 0$ for all $i$. The reduced form of $w$ decomposes as
$w = \alpha\beta\gamma$
where length$(\beta)\geq 1$ and $\alpha \in \{ \Delta y, y^2\Delta\}$, $\gamma\in \{ y\Delta, \Delta y^2\}$.
\end{propst}

\begin{proof} We proceed by induction. For $n= 1$ is clearly true. Suppose that it is valid for $n=m$, we will prove it for $n= m+1$.

$$w = \prod_{i=1}^{m+1} (y^2\Delta)^{2e_i}(\Delta y^2)^{2f_i}$$

$$= w' (y^2\Delta)^{2e_{m+1}}(\Delta y^2)^{2f_{m+1}}$$ by induction hypothesis we know that the reduced form of $w'$ is $w' =\alpha \beta \gamma$ where $\gamma = y\Delta$ or $\gamma\Delta y^2$. We also know that $w'' = (y^2\Delta )^{2e_{m+1}}(\Delta y^2)^{2f_{m+1}}$
has the form $\alpha' \beta ' \gamma '$. So, $w = (\alpha \beta \gamma)(\alpha ' \beta ' \gamma ')$.

\vskip 10pt

\begin{table}[!ht]
    \centering
    \begin{tabular}
    {c|c|c|}
      $\gamma \times \alpha '$  & $\alpha' = \Delta y$ &  $\alpha' = y^2\Delta$\\
     \hline
     $\gamma = y\Delta $ & $y^2$ & $\Delta y^2\Delta y$ \\
     \hline
     $ \gamma = \Delta y^2$ & $y\Delta y^2 \Delta$ & $\Delta y \Delta$\\
     \hline
    \end{tabular}
    \caption{Simplifications of the possible products between $\gamma$ and $\alpha'$}
    \label{tab:table_1a}
\end{table}

But, by analyzing  all the four possibilities for $\gamma$ and $\alpha '$ we can see that $\gamma\alpha '$ at most reduces to $y^2$, but even on this case, $\beta y^2\beta '$ does not reduce any further because the $y^2$ reduction occurs in the case where $\gamma = y\Delta$ and $\alpha ' =\Delta y$, so $\beta$ must ends with  $\Delta$ and $\beta '$ starts with $\Delta$. So $\beta y^2\beta'$ is in its reduced form. So, $ w = \alpha \beta y^2 \beta' \gamma'$ decomposes as we say it would. This concludes the proof.

\end{proof}

\begin{thm} Let $w= \prod_{i=1}^n (y^2\Delta)^{2e_i}(\Delta y^2)^{2f_i} $ be a word with $e_i,f_i \neq 0 $ 
for all $i$. Then $w$ is cyclically reduced equal to
$(y^2\Delta)^{2k}$ or $(\Delta y)^{2k}$ if and only if $e_i = f_i = \pm 1 $ for all $i$.
\end{thm}

\begin{proof} We already proved this for the case
$n=1$. We will show it when $n >1$.
So, $w= (y^2\Delta)^{2e_1}(\Delta y^2)^{2f_1}w'$.
As we proved,  $w' = \alpha \beta \gamma$ where 
$\alpha = (y^2\Delta)^{\pm 1}$
and $\gamma = (\Delta y^2)^{\pm1}$ is in reduced form. Observe that
$$ \eta = (y^2\Delta)^{2e_i}(\Delta y^2)^{2f_i}$$
reduces in four different forms depending on the signs of $e_i, f_i$.  

\begin{table}[H]
    \centering
    \begin{tabular}
    {c|c|c|}
     %\hline
       & $e_i > 0 $ & $e_i < 0$ \\
     \hline
     $f_i > 0$ & $(y^2\Delta)^a y (\Delta y^2)^b$ & $(\Delta y)^{a+1} (\Delta y^2 )^{b+1}$ \\
     \hline
     $f_i < 0$ & $( y^2 \Delta)^{a+1}(y\Delta)^{b+1}$ & $(\Delta y )^a \Delta y^2 \Delta (y\Delta)^b$\\
     \hline
    \end{tabular}
    %\vspace{1em}
    \caption{ Reduce form of $\eta =  (y^2\Delta)^{2e_i}(\Delta y^2)^{2f_i}$. For simplification we set $a = 2|e_i| -1$ and $b = 2|f_i| -1$. Notice that $a, b > 0$ .}
    \label{tab:table_2a}
\end{table}

By looking at the options in \ref{tab:table_2a}, we notice that the reduced form of $\eta$ has $y$ and $y^2$ on always. So, if we want it to be cyclically reduced equal to $(\Delta y)^{2k}$ or $(y^2\Delta)^{2k}$ we need only to have one, so we must reduce $\eta$ cyclically, $w'$ has the form $$w' = (y^2\Delta)^{\pm 1}\beta (\Delta y^2)^{\pm 1}$$.
Analyzing all the possibilities, we observe that when $e_i$ and $f_i$ have opposite  signs, $y$, and $y^2$ will survive after a cyclical reduction of $w$. Because at most one $y$ or $y^2$ is canceled (or transformed) on each side of $\eta$ (left and right)
as observed in Table \ref{tab:table_1a}.
But in \ref{tab:table_2a}, we can observe that $\eta$ has two $y$'s and two $y^2$'s when
$e_if_i< 0$ and when $e_if_i > 0$ it has one $y$ in the middle surrounded by $y^2$
on both sides, as we can only cancel at most two $y^2$
the only way that this will work is when $e_i = f_i =1$.
Analogously, if $e_i f_i < 0$
can be reduced cyclically eliminating the $y$'s at the sides of $\eta$  and leaving only $y^2$'s if $e_i = f_i = -1$.
Observe that, when $e_i = f_i = 1$ the reduced word
has only $y$'s and when $e_i = f_i = -1$ the reduced word of $\eta$ has only $y^2$'s.
So, we apply this argument for $(y^2\Delta)^{2e_i}(\Delta y^2)^{2f_i}$ to obtain that $e_i = f_i = \pm 1$.

As we want only $y$'s or $y^2$'s we conclude that either $e_i = f_i = 1$
for all $i$ or $e_i = f_i = -1$ for all $i$.

\end{proof}

\begin{thm}
 Let $\beta$ be the following pure $3$-braid
$$ \beta =  \prod _{i=1}^n\sigma_1^{2e_i}\sigma_2^{2f_i}(\sigma_1\sigma_2\sigma_1)^{2e}$$

\noindent where $e_i, f_i, e$ any integers.
Then $\hat \beta$  is hyperbolic except if we have: 

i) Some of the braids in Theorem $4.2$

ii) $e_i = 0$ or $f_i = 0$ for all $i$ and $e = 0$ in which case $\hat \beta$ is splittable.

iii)  $e_i = f_i = 1$ for all $i$

iv)  $e_i = f_i = -1$  for all $i$
  
\end {thm}

\begin{proof} We show that $\hat \beta$ does not have an  essential annulus or sphere.
For the annulus, it is enough to analyze when its boundary is in different or in the same component of the link.
It is enough to take a neighborhood of the annulus and the link, and then we get an essential torus.
If there is an essential sphere then the link is splittable. 
\end{proof}

%\begin{figure}[ht]
%\includegraphics[width=4cm]{trenzaanillo1.eps} 
%\end{figure}

%\begin{figure}[h]
%\includegraphics[width=10cm]{fighexatangle2.pdf} 
%\end{figure}

%\begin{table}[]
 %   \centering
 %   \include{TrivialGroups1/table_factorized}
  %  \caption{Caption}
   % \label{tab:my_label}
%\end{table}
%\begin{figure}[h]
%\includegraphics[width=10cm]{fighexatangle3.pdf} 
%\end{figure}

%\begin{figure}[h]
%\includegraphics[width=6cm]{hexatangle(3).png} 
%\end{figure}

%%%%%%%%%%%%%%%%%%%%%%%%%%%%%%%

\textbf{Acknowledgement.} The first author is supported by a fellowship \textit{Investigadoras por M\'exico} from CONAHCyT. Research partially supported by the grant PAPIIT-UNAM 116720.

%Notacion para referencias: Figure \ref{Fig:braids4,6} completes the proof.
%%%%%%%%BIBLIO

\bibliographystyle{alpha}
\bibliography{TrivialGroups1/bibliography}

\end{document}

%% file: tables/table_1.tex
 \begin{tabular}{|r|p{10cm}|}
 \hline
 \textrm{No.}& $r_1,\ r_2,\ r_3$\\

 \hline
  1 & $x_1^{-1},\ $ $x_2^{-1}x_3^{-1}x_2^{-1},\ $ $x_3^{-1}x_2^{-1}$\\

 \hline
  2 & $x_2^{-1}x_1^{-1},\ $ $x_2^{-2}x_1^{-1},\ $ $x_3^{-1}$\\

 \hline
  3 & $x_3^{-1}x_2^{-1}x_1^{-1},\ $ $x_3^{-1}x_2^{-1}x_3^{-1}x_2^{-1}x_1^{-1},\ $ $x_3^{-2}x_2^{-1}x_3^{-1}x_2^{-1}x_1^{-1}$\\

 \hline
  4 & $x_1^{-1}x_2^{-1}x_1^{-1}x_3^{-1}x_2^{-1}x_1^{-1},\ $ $x_2^{-1}x_1^{-1}x_3^{-1}x_2^{-1}x_1^{-1},\ $ $x_3^{-1}x_2^{-1}x_1^{-1}$\\

 \hline
  5 & $x_1^{-1},\ $ $x_3,\ $ $x_2x_3$\\

 \hline
  6 & $x_1^{-1},\ $ $x_2x_3^{-1}x_2^{-1},\ $ $x_3^{-1}x_2^{-1}$\\

 \hline
  7 & $x_3^{-1}x_2^{-1}x_1^{-1},\ $ $x_3^{-1}x_2^{-1}x_3^{-1}x_2^{-1}x_1^{-1},\ $ $x_2^{-1}x_3^{-1}x_2^{-1}x_1^{-1}$\\

 \hline
  8 & $x_2x_3^{-1}x_2^{-1}x_1^{-1},\ $ $x_1x_2x_3^{-1}x_2^{-1}x_1^{-1},\ $ $x_3^{-1}x_2^{-1}x_1^{-1}$\\

 \hline
  9 & $x_1x_2x_3,\ $ $x_3^{-1}x_2^{-1}x_1x_2x_3,\ $ $x_3^{-2}x_2^{-1}x_1x_2x_3$\\

 \hline
  10 & $x_1^{-1}x_3,\ $ $x_3,\ $ $x_1x_2x_3$\\

 \hline
  11 & $x_3^{-1}x_2^{-1}x_1^{-1},\ $ $x_1^{-1},\ $ $x_3^{-1}x_1^{-1}$\\

 \hline
  12 & $x_1x_2^{-1}x_1^{-1}x_3^{-1}x_2^{-1}x_1^{-1},\ $ $x_2^{-1}x_1^{-1}x_3^{-1}x_2^{-1}x_1^{-1},\ $ $x_3^{-1}x_2^{-1}x_1^{-1}$\\

 \hline
  13 & $x_1^{-1}x_2x_3x_2^{-1}x_3^{-1}x_2^{-1}x_1^{-1},\ $ $x_2^{-1}x_3^{-1}x_2^{-1}x_1^{-1},\ $ $x_3^{-1}x_2^{-1}$\\

 \hline
  14 & $x_1^{-1},\ $ $x_2^{-1},\ $ $x_3^{-1}$\\

 \hline
  15 & $x_2x_3x_2^{-1}x_3^{-1}x_2^{-1}x_1^{-1},\ $ $x_2^{-2}x_3^{-1}x_2^{-1}x_1^{-1},\ $ $x_3^{-1}x_2^{-1}$\\

 \hline
  16 & $x_1^{-1}x_3^{-1}x_2^{-1}x_1^{-1},\ $ $x_2^{-1}x_3^{-1}x_2^{-1}x_1^{-1},\ $ $x_3^{-1}x_2^{-1}x_1^{-1}$\\

 \hline
  17 & $x_3^{-1}x_2^{-1}x_1^{-1},\ $ $(x_2^{-1}x_3^{-1})(x_2^{-1}x_3^{-1})x_2^{-1}x_1^{-1},\ $ $x_3^{-1}x_2^{-1}x_3^{-1}x_2^{-1}x_1^{-1}$\\

 \hline
  18 & $x_2x_3x_2^{-1}x_3^{-1}x_2^{-1}x_1^{-1},\ $ $x_2^{-1}x_3^{-1}x_2^{-1}x_1^{-1},\ $ $x_3^{-2}x_2^{-1}$\\

 \hline
  19 & $x_3^{-1}x_2^{-1}x_1^{-1},\ $ $x_2^{-1}x_3^{-1}x_2^{-1}x_1^{-1},\ $ $x_3^{-2}x_2^{-1}x_1^{-1}$\\

 \hline
  20 & $x_2^{-1}x_1^{-1}x_3^{-1}x_2^{-1}x_1^{-1},\ $ $x_2^{-2}x_1^{-1}x_3^{-1}x_2^{-1}x_1^{-1},\ $ $x_3^{-1}x_2^{-1}x_1^{-1}$\\
\hline
\end{tabular}

%% file: tables/table_2.tex
 \begin{tabular}{|r|p{14cm}|}
 \hline
 \textrm{No.}& $r_1,\ r_2,\ r_3$\\

 \hline
  1 & $x_1^2(x_2x_3)^2x_2(x_2x_3)^{-2},\ $ $(x_2x_3)^{-2}x_1(x_2x_3)^2x_2(x_2x_3)^{-2},\ $ $x_3^{-2}x_2^{-1}x_3^{-1}x_2^{-1}$\\

 \hline
  2 & $x_1^{-2}x_3^{-1}x_2^{-1}x_3x_1^{-1},\ $ $x_3x_1^{-1},\ $ $x_3x_2x_3$\\

 \hline
  3 & $x_1^2x_2x_3,\ $ $x_2^{-2}x_1x_2x_3,\ $ $x_3^{-1}x_1x_2x_3$\\

 \hline
  4 & $x_3^{-1}x_2^{-1}x_3x_2x_3,\ $ $x_2^{-2}x_3x_2x_3,\ $ $x_2x_3x_1x_2x_3$\\

 \hline
  5 & $x_1^{-2}x_3^{-1}x_2^{-1}x_1^{-1},\ $ $x_2x_3^{-1}x_2^{-1}x_1^{-1},\ $ $x_2^{-1}x_1^{-1}$\\

 \hline
  6 & $x_1(x_2x_3)^2x_2(x_2x_3)^{-3}x_1^{-1},\ $ $x_2(x_2x_3)^{-2}x_1(x_2x_3)^{2}x_2(x_2x_3)^{-3}x_1^{-1},\ $ $(x_2x_3)(x_2x_3)^{-2}x_1^{-1}$\\

 \hline
  7 & $x_1^2x_2x_3x_2x_3^{-1}x_2^{-1},\ $ $x_3^{-1}x_2^{-1}x_1x_2x_3x_2x_3^{-1}x_2^{-1},\ $ $x_3^{-3}x_2^{-1}$\\

 \hline
  8 & $x_1^{-1}x_3^{-1}x_2^{-1}x_3x_1^{-1}x_3^{-1}x_2^{-1}x_3x_1^{-1},\ $ $x_3x_1^{-1}x_3^{-1}x_2^{-1}x_3x_1^{-1},\ $ $x_3x_2x_3$\\

 \hline
  9 & $x_1^2x_2x_3,\ $ $x_2^{-1}x_1x_2x_3,\ $ $x_3^{-2}x_1x_2x_3$\\

 \hline
  10 & $x_3^{-1}x_2^{-1}x_3x_1^{-1}x_3^{-1}x_2^{-1}x_3x_2x_3,\ $ $x_2^{-1}x_3x_1^{-1}x_3^{-1}x_2^{-1}x_3x_2x_3,\ $ $x_2x_3x_1x_2x_3$\\

 \hline
  11 & $x_1^{-1}x_3^{-1}x_2^{-1}x_1^{-1}x_3^{-1}x_2^{-1}x_1^{-1},\ $ $x_2x_3^{-1}x_2^{-1}x_1^{-1}x_3^{-1}x_2^{-1}x_1^{-1},\ $ $x_2^{-1}x_1^{-1}x_3^{-1}x_2^{-1}x_1^{-1}$\\

 \hline
  12 & $x_1x_2x_3x_2(x_2x_3)^{-2}(x_1x_2x_3x_1)^{-1},$   $x_2(x_2x_3)^{-1}x_1x_2x_3x_2(x_2x_3)^{-2}(x_1x_2x_3x_2)^{-1},\ $ $(x_2x_3)^{-2}x_1^{-1}x_3^{-1}x_2^{-1}x_1^{-1}$\\

 \hline
  13 & $x_1x_2x_3x_2x_3^{-1}x_2^{-1}x_1x_2x_3,\ $ $x_2^{-2}x_3^{-1}x_2^{-1}x_1x_2x_3x_2x_3^{-1}x_2^{-1}x_1x_2x_3,\ $ $x_3^{-1}x_2^{-1}x_1x_2x_3$\\

 \hline
  14 & $x_2x_3,\ $ $x_2^{-2}x_1x_2x_3,\ $ $x_3x_1x_2x_3$\\

 \hline
  15 & $x_1(x_2x_3)^2x_2(x_2x_3)^{-2}x_1x_2x_3,\ $ $(x_3x_2)^{-2}x_2^{-1}x_1(x_2x_3)^2x_2(x_2x_3)^{-2}x_1x_2x_3,\ $ $(x_2x_3)^{-2}x_1x_2x_3$\\

 \hline
  16 & $x_1^{-1}x_2x_3,\ $ $x_2^{-1}x_1x_2x_3,\ $ $x_3x_1x_2x_3$\\

 \hline
  17 & $(x_2x_3)^2x_2(x_2x_3)^{-2},\ $ $(x_2x_3)^{-2}x_1(x_2x_3)^{2}x_2(x_2x_3)^{-2},\ $ $x_2^{-1}x_3^{-1}x_2^{-1}$\\

 \hline
  18 & $x_1^{-1}x_2x_3x_2x_3^{-1}x_2^{-1},\ $ $x_3^{-1}x_2^{-1}x_1x_2x_3x_2x_3^{-1}x_2^{-1},\ $ $x_2^{-1}$\\

 \hline
  19 & $(x_3^{-1}x_2^{-1}x_3x_1^{-1})(x_3^{-1}x_2^{-1}x_3x_1^{-1})x_3^{-1}x_2^{-1}x_1^{-1},\ $ $x_2(x_3x_1^{-1}x_3^{-1}x_2^{-1})(x_3x_1^{-1}x_3^{-1}x_2^{-1})x_1^{-1},\ $ $x_1^{-1}$\\

 \hline
  20 & $x_1x_3^{-1}x_2^{-1}x_1^{-1},\ $ $x_2x_3^{-1}x_2^{-1}x_1^{-1},\ $ $x_3^{-3}x_2^{-1}x_1^{-1}$\\

 \hline

\end{tabular}

%% file: tables/table_3.tex
 \begin{tabular}{|r|p{14cm}|}
\hline
  1 & $(x_2x_3)^{\gamma}x_2(x_2x_3)^{-\gamma},\ $ $x_2(x_2x_3)^{-\gamma}x_1(x_2x_3)^{\gamma}x_2(x_2x_3)^{-\gamma},\ $ $x_3^{\gamma-1}(x_2x_3)^{-\gamma}$\\

 \hline
  2 & $x_1^{-\gamma}(x_1x_2x_3x_2x_3^{-1}x_2^{-1})^{\gamma-1},\ $ $x_2x_3^{-1}x_2^{-1}(x_1x_2x_3x_2x_3^{-1}x_2^{-1})^{\gamma-1},\ $ $x_2^{-1}$\\

 \hline
  3 & $x_2x_3,\ $ $x_2^{-\gamma+1}x_3x_1x_2x_3,\ $ $x_3^{\gamma-1}x_2x_3x_1x_2x_3$\\

 \hline
  4 & $x_1(x_1x_2x_3x_2x_3^{-1}x_2^{-1})^{\gamma-1}x_1x_2x_3,\ $ $x_2^{-\gamma}x_3^{-1}x_2^{-1}(x_1x_2x_3x_2x_3^{-1}x_2^{-1})^{\gamma-1}x_1x_2x_3,\ $ $x_3^{-1}x_2^{-1}x_1x_2x_3$\\

 \hline
  5 & $x_1^{-\gamma}(x_1x_2x_3)^{\gamma-1},\ $ $x_3(x_1x_2x_3)^{\gamma-1},\ $ $x_3x_2x_3(x_1x_2x_3)^{\gamma-1}$\\

 \hline
  6 & $x_1^2(x_2x_3)^{\gamma}x_2(x_2x_3)^{-\gamma}(x_1x_2x_3)^{\gamma-1},\ $ $x_2^{-1}(x_2x_3)^{-\gamma}x_1(x_2x_3)^{\gamma}x_2(x_2x_3)^{-\gamma}(x_1x_2x_3)^{\gamma-1},\ $ $(x_2x_3)^{-\gamma}(x_1x_2x_3)^{\gamma-1}$\\

 \hline
  7 & $x_1(x_2x_3)^{-\gamma+1}x_2^{-1}(x_2x_3)^{\gamma-1}x_1^{-1},\ $ $(x_2x_3)^{\gamma-1}x_1^{-1},\ $ $x_3^{-\gamma}(x_2x_3)^{\gamma-1}$\\

 \hline
  8 & $x_1^{\gamma-1}(x_1x_3^{-1}x_2x_3)^{-\gamma},\ $ $x_2^2x_3(x_1x_3^{-1}x_2x_3)^{-\gamma},\ $ $x_3^{-1}x_2x_3$\\

 \hline
  9 & $x_1x_3^{-1}x_2^{-1}x_1^{-1},\ $ $x_2^{\gamma-1}x_1^{-1},\ $ $x_3^{-\gamma}x_1^{-1}$\\

 \hline
  10 & $x_1(x_1x_3^{-1}x_2x_3)^{-\gamma}x_3^{-1}x_2^{-1}x_1^{-1},\ $ $x_2^{\gamma}x_3(x_1x_3^{-1}x_2x_3)^{-\gamma}x_3^{-1}x_2^{-1}x_1^{-1},\ $ $x_1^{-1}$\\

 \hline
  11 & $x_1^{\gamma-1}(x_1x_2x_3)^{-\gamma},\ $ $x_2^2x_3(x_1x_2x_3)^{-\gamma},\ $ $x_3^{-1}x_2x_3(x_1x_2x_3)^{-\gamma}$\\

 \hline
  12 & $x_1(x_2x_3)^{-\gamma+1}x_2^{-1}(x_2x_3)^{\gamma-1}x_1^{-1}(x_1x_2x_3)^{-\gamma},\ $ $(x_2x_3)^{\gamma-1}x_1^{-1}(x_1x_2x_3)^{-\gamma},\ $ $(x_2x_3)^{\gamma-1}(x_1x_2x_3)^{-\gamma}$\\

 \hline
  13 & $(x_2x_3)^{-\gamma+1}x_2^{-1}(x_2x_3)^{\gamma},\ $ $x_2^{-\gamma-1}(x_2x_3)^{\gamma},\ $ $x_3(x_2x_3)^{\gamma-1}x_1x_2x_3$\\

 \hline
  14 & $x_1^{\gamma}x_2x_1x_2x_3,\ $ $x_2^{-\gamma}(x_1x_2)(x_1x_2)x_3,\ $ $x_3^{-1}x_1x_2x_3$\\

 \hline
  15 & $x_1(x_2x_3)^{\gamma}x_2(x_2x_3)^{-\gamma}x_3^{-1}x_2^{-1}x_1^{-1},\ $ $x_2^{\gamma-1}(x_2x_3)^{-\gamma}x_1(x_2x_3)^{\gamma}x_2(x_2x_3)^{-\gamma}x_3^{-1}x_2^{-1}x_1^{-1},\ $ $x_3(x_2x_3)^{-\gamma}x_3^{-1}x_2^{-1}x_1^{-1}$\\

 \hline
  16 & $x_1^{-\gamma+1}x_2x_3^{-1}x_2^{-1}x_1^{-1},\ $ $x_2^{\gamma-1}x_1x_2x_3^{-1}x_2^{-1}x_1^{-1},\ $ $x_2^{-1}x_1^{-1}$\\

 \hline
  17 & $x_1^{\gamma}(x_2x_3)^{\gamma}x_2(x_2x_3)^{-\gamma}x_1x_2x_3,\ $ $(x_2x_3)^{-\gamma}x_1(x_2x_3)^{\gamma}x_2(x_2x_3)^{-\gamma}x_1x_2x_3,\ $ $x_3^{-1}(x_2x_3)^{-\gamma}x_1x_2x_3$\\

 \hline
  18 & $x_1^{-\gamma}(x_2x_3)^{-\gamma+1}x_2^{-1}(x_2x_3)^{\gamma},\ $ $x_2^{-1}(x_2x_3)^{\gamma},\ $ $x_3(x_2x_3)^{\gamma-1}x_1x_2x_3$\\

 \hline
  19 & $(x_1x_3^{-1}x_2x_3)^{-\gamma}(x_1x_2x_3)^{\gamma-1},\ $ $x_3(x_1x_3^{-1}x_2x_3)^{-\gamma}(x_1x_2x_3)^{\gamma-1},\ $ $x_3x_2x_3(x_1x_2x_3)^{\gamma-1}$\\

 \hline
  20 & $x_1^2x_2(x_1x_2x_3)^{\gamma-1},\ $ $x_2^{-1}x_1x_2(x_1x_2x_3)^{\gamma-1},\ $ $x_3^{-\gamma}(x_1x_2x_3)^{\gamma-1}$\\
\hline
\end{tabular}

%% file: tables/table_4.tex
\begin{tabular}{|r|p{15cm}|}
\hline
\textrm{No.}& $r_1,\ r_2,\ r_3$\\
\hline
  1 & $(x_2x_3)^{\gamma}x_2(x_2x_3)^{-\gamma}x_1(x_2x_3)^{\gamma}x_2(x_2x_3)^{-\gamma},\ $ $(x_2(x_2x_3)^{-\gamma}x_1(x_2x_3)^{\gamma})^2x_2(x_2x_3)^{-\gamma},\ $ $x_3(x_2x_3)^{-\gamma}$\\

 \hline
  2 & $x_1^{-\gamma+1}x_2x_3x_2x_3^{-1}x_2^{-1},\ $ $x_2x_3^{-1}x_2^{-1}x_1x_2x_3x_2x_3^{-1}x_2^{-1},\ $ $x_3x_2^{-1}$\\

 \hline
  3 & $x_2x_3x_1x_2x_3,\ $ $x_2^{-\gamma+1}(x_3x_1x_2)(x_3x_1x_2)x_3,\ $ $x_3(x_2x_3x_1)(x_2x_3x_1)x_2x_3$\\

 \hline
  4 & $x_1^2x_2x_3x_2x_3^{-1}x_2^{-1}(x_1x_2x_3)(x_1x_2x_3),\ $ $x_2^{-\gamma}x_3^{-1}x_2^{-1}x_1x_2x_3x_2x_3^{-1}x_2^{-1}(x_1x_2x_3)(x_1x_2x_3),\ $ $x_3^{-1}x_2^{-1}(x_1x_2x_3)(x_1x_2x_3)$\\

 \hline
  5 & $x_1^{-\gamma+1}x_2x_3,\ $ $x_3x_1x_2x_3,\ $ $x_3^2x_2x_3x_1x_2x_3$\\

 \hline
  6 & $x_1^2((x_2x_3)^{\gamma}x_2(x_2x_3)^{-\gamma}x_1)((x_2x_3)^{\gamma}x_2(x_2x_3)^{-\gamma}x_1)x_2x_3,\ $ $x_2^{-1}((x_2x_3)^{-\gamma}x_1(x_2x_3)^{\gamma}x_2)((x_2x_3)^{-\gamma}x_1(x_2x_3)^{\gamma}x_2)(x_2x_3)^{-\gamma}x_1x_2x_3,\ $ $(x_2x_3)^{-\gamma}x_1x_2x_3$\\

 \hline
  7 & $x_1^2x_3^{-1}x_2^{-1}x_3x_1^{-1},\ $ $x_2x_3x_1^{-1},\ $ $x_3^{-\gamma}x_2x_3$\\

 \hline
  8 & $x_1(x_1x_3^{-1}x_2^{-1}x_3^{-1}x_2x_3x_2x_3)^{-\gamma},\ $ $x_2^2x_3x_2x_3(x_1x_3^{-1}x_2^{-1}x_3^{-1}x_2x_3x_2x_3)^{-\gamma},\ $ $x_3^{-1}x_2x_3x_2x_3$\\

 \hline
  9 & $x_1^2(x_1x_2x_3)^{-1},\ $ $x_2x_1^{-1},\ $ $x_3^{-\gamma}x_1^{-1}$\\

 \hline
  10 & $x_1(x_1(x_3x_2x_3)^{-1}(x_2x_3)^2)^{-\gamma}(x_1x_2x_3)^{-1},\ $ $x_2^2x_3x_2x_3(x_1(x_3x_2x_3)^{-1}(x_2x_3)^2)^{-\gamma}(x_1x_2x_3)^{-1},\ $ $x_2x_3x_1^{-1}$\\

 \hline
  11 & $x_1(x_1x_2x_3)^{-\gamma},\ $ $x_2^3x_3(x_1x_2x_3)^{-\gamma},\ $ $x_3^{-1}x_2x_3(x_1x_2x_3)^{-\gamma}$\\

 \hline
  12 & $x_1x_3^{-1}x_2^{-1}x_3x_1^{-1}(x_1x_2x_3)^{-\gamma},\ $ $x_2^2x_3x_1^{-1}(x_1x_2x_3)^{-\gamma},\ $ $x_2x_3(x_1x_2x_3)^{-\gamma}$\\

 \hline
  13 & $x_3^{-1}x_2^{-1}x_3x_2x_3x_1x_2x_3,\ $ $x_2^{-\gamma}x_3x_2x_3x_1x_2x_3,\ $ $x_3(x_2x_3x_1)(x_2x_3x_1)x_2x_3$\\

 \hline
  14 & $x_1^2x_2(x_1x_2x_3)(x_1x_2x_3),\ $ $x_2^{-\gamma}(x_1x_2)(x_1x_2)x_3x_1x_2x_3,\ $ $x_3^{-1}x_1x_2x_3x_1x_2x_3$\\

 \hline
  15 & $(x_1(x_2x_3)^{\gamma}x_2(x_2x_3)^{-\gamma})(x_1(x_2x_3)^{\gamma}x_2(x_2x_3)^{-\gamma})(x_1x_2x_3)^{-1},\ $ $(x_2(x_2x_3)^{-\gamma}x_1(x_2x_3)^{\gamma})(x_2(x_2x_3)^{-\gamma}x_1(x_2x_3)^{\gamma})x_2(x_2x_3)^{-\gamma}(x_1x_2x_3)^{-1},\ $ $x_3(x_2x_3)^{-\gamma}(x_1x_2x_3)^{-1}$\\

 \hline
  16 & $x_1^{-\gamma+1}x_2(x_1x_2x_3)^{-1},\ $ $x_2x_1x_2(x_1x_2x_3)^{-1},\ $ $x_3x_2^{-1}x_1^{-1}$\\

 \hline
  17 & $x_1^2((x_2x_3)^{\gamma}x_2(x_2x_3)^{-\gamma}x_1)((x_2x_3)^{\gamma}x_2(x_2x_3)^{-\gamma}x_1)x_2x_3,\ $ $((x_2x_3)^{-\gamma}x_1(x_2x_3)^{\gamma}x_2)((x_2x_3)^{-\gamma}x_1(x_2x_3)^{\gamma}x_2)(x_2x_3)^{-\gamma}x_1x_2x_3,\ $ $x_3^{-1}(x_2x_3)^{-\gamma}x_1x_2x_3$\\

 \hline
  18 & $x_1^{-\gamma}x_3^{-1}x_2^{-1}x_3x_2x_3,\ $ $x_3x_2x_3,\ $ $x_3^2x_2x_3x_1x_2x_3$\\

 \hline
  19 & $(x_1x_3^{-1}x_2^{-1}x_3^{-1}x_2x_3x_2x_3)^{-\gamma}x_1x_2x_3,\ $ $x_3x_2x_3(x_1x_3^{-1}x_2^{-1}x_3^{-1}x_2x_3x_2x_3)^{-\gamma}x_1x_2x_3,\ $ $(x_3x_2)(x_3x_2)x_3x_1x_2x_3$\\

 \hline
  20 & $x_1^3x_2x_1x_2x_3,\ $ $x_2^{-1}(x_1x_2)(x_1x_2)x_3,\ $ $x_3^{-\gamma}x_1x_2x_3$\\

 \hline
\end{tabular}

%% file: tables/table_5.tex
\begin{tabular}{|r|p{14cm}|}
 \hline
 \textrm{No.}& $r_1,\ r_2,\ r_3$\\

 \hline
  1 & $x_3^{-1}x_2^{-1}x_3^{-1}x_2x_3x_1^{-1},\ $ $x_2^{-\beta+1}x_3x_2x_3x_1x_3^{-1}x_2^{-1}x_3^{-1}x_2x_3x_1^{-1},\ $ $x_3^{-2}x_2x_3x_1^{-1}$\\

 \hline
  2 & $x_1^2x_2x_3x_2^{-1}x_3^{-1}x_2^{-1}x_1^{-1}x_2x_3x_2^{-1}(x_3^{-1}x_2^{-1}x_1^{-1})(x_3^{-1}x_2^{-1}x_1^{-1}),\ $ $x_2^{-\beta-1}x_3^{-1}x_2^{-1}x_1^{-1}x_2x_3x_2^{-1}(x_3^{-1}x_2^{-1}x_1^{-1})(x_3^{-1}x_2^{-1}x_1^{-1}),\ $ $x_2^{-1}x_3^{-1}x_2^{-1}x_1^{-1}$\\

 \hline
  3 & $x_1^{-1}(x_2x_3)^{\beta}x_2^{-1}(x_2x_3)^{-\beta+1},\ $ $x_2(x_2x_3)^{-\beta+1},\ $ $x_3^{-2}(x_2x_3)^{-\beta}x_1x_2x_3$\\

 \hline
  4 & $x_1^{-\beta}x_2x_3x_2^{-1}x_3^{-1}x_2^{-1}x_1^{-1}x_2x_3x_2^{-1},\ $ $x_2x_3^{-1}x_2^{-1}x_1^{-1}x_2x_3x_2^{-1},\ $ $x_3^{-2}x_2^{-1}x_1x_2x_3$\\

 \hline
  5 & $x_1^2(x_2x_3)^{\beta}x_2^{-1}(x_2x_3)^{-\beta}(x_1^{-1}x_3^{-1}x_2^{-1})(x_1^{-1}x_3^{-1}x_2^{-1})x_1^{-1},\ $ $x_2^{-2}(x_2x_3)^{-\beta}(x_1^{-1}x_3^{-1}x_2^{-1})(x_1^{-1}x_3^{-1}x_2^{-1})x_1^{-1},\ $ $x_3(x_2x_3)^{-\beta}(x_3^{-1}x_2^{-1}x_1^{-1})(x_3^{-1}x_2^{-1}x_1^{-1})$\\

 \hline
  6 & $x_1^{-\beta+1}x_3^{-1}x_2^{-1}x_3^{-1}x_2x_3x_1^{-1}x_3^{-1}x_2^{-1}x_1^{-1},\ $ $x_3x_2x_3x_1x_3^{-1}x_2^{-1}x_3^{-1}x_2x_3x_1^{-1}x_3^{-1}x_2^{-1}x_1^{-1},\ $ $x_3^{-1}x_2x_3x_1^{-1}x_3^{-1}x_2^{-1}x_1^{-1}$\\

 \hline
  7 & $x_1(x_2x_3)^2(x_2^{-1}x_3^{-1})(x_2^{-1}x_3^{-1})x_2^{-1}x_1^{-1}x_3^{-1}x_2^{-1}x_1^{-1},\ $ $x_2^{-\beta-1}(x_3^{-1}x_2^{-1})(x_3^{-1}x_2^{-1})x_1^{-1}x_3^{-1}x_2^{-1}x_1^{-1},\ $ $x_3(x_2^{-1}x_3^{-1})(x_2^{-1}x_3^{-1})x_2^{-1}x_1^{-1}$\\

 \hline
  8 & $x_1^{-1}x_3^{-1}x_2x_3x_1x_3^{-1}x_1^{-1},\ $ $x_2^{-\beta+1}x_3x_1x_3^{-1}x_2x_3x_1x_3^{-1}x_1^{-1},\ $ $x_3^{-1}x_1^{-1}$\\

 \hline
  9 & $x_1(x_2x_3)^{\beta}x_2^{-1}(x_2x_3)^{-\beta}x_1^{-1}x_3^{-1}x_2^{-1}x_1^{-1},\ $ $x_2^{-3}(x_2x_3)^{-\beta}x_1^{-1}x_3^{-1}x_2^{-1}x_1^{-1},\ $ $x_3^2(x_2x_3)^{-\beta}x_3^{-1}x_2^{-1}x_1^{-1}$\\

 \hline
  10 & $x_1^{-\beta+1}x_3^{-1}x_2x_3x_1x_3^{-1}x_1^{-1},\ $ $x_2^{-1}x_3x_1x_3^{-1}x_2x_3x_1x_3^{-1}x_1^{-1},\ $ $x_3^{-1}x_1^{-1}$\\

 \hline
  11 & $x_1^{-2}(x_2x_3)^{\beta}x_2^{-1}(x_2x_3)^{-\beta+1}x_1x_2x_3,\ $ $(x_2x_3)^{-\beta+1}x_1x_2x_3,\ $ $x_3^{-1}(x_2x_3)^{-\beta}(x_1x_2x_3)(x_1x_2x_3)$\\

 \hline
  12 & $x_1^{-\beta}(x_2x_3)^2x_2^{-1}x_3^{-1}x_2^{-1}x_1x_2x_3,\ $ $x_3^{-1}x_2^{-1}x_1x_2x_3,\ $ $x_3^{-2}x_2^{-1}x_3^{-1}x_2^{-1}(x_1x_2x_3)(x_1x_2x_3)$\\

 \hline
  13 & $x_1^{-1}(x_2x_3)^2x_2^{-1}x_3^{-1}x_2^{-1},\ $ $x_2x_3^{-1}x_2^{-1},\ $ $x_3^{-\beta-1}x_2^{-1}x_3^{-1}x_2^{-1}x_1x_2x_3$\\

 \hline
  14 & $x_1^{-2}(x_1x_2x_3x_2x_3^{-1}x_2^{-1})^{-\beta}x_1x_2x_3,\ $ $x_2^2x_3^{-1}x_2^{-1}(x_1x_2x_3x_2x_3^{-1}x_2^{-1})^{-\beta}x_1x_2x_3,\ $ $x_3^{-2}x_2^{-1}x_1x_2x_3$\\

 \hline
  15 & $x_3^{-1}x_2^{-1}x_3^{-1}x_2x_3x_1^{-1},\ $ $x_2^{-1}x_3x_2x_3x_1x_3^{-1}x_2^{-1}x_3^{-1}x_2x_3x_1^{-1},\ $ $x_3^{-\beta}x_2x_3x_1^{-1}$\\

 \hline
  16 & $x_1^2(x_1x_2x_3x_2x_3^{-1}x_2^{-1})^{-\beta}x_3^{-1}x_2^{-1}x_1^{-1},\ $ $x_2^{-2}x_3^{-1}x_2^{-1}(x_1x_2x_3x_2x_3^{-1}x_2^{-1})^{-\beta}x_3^{-1}x_2^{-1}x_1^{-1},\ $ $x_2^{-1}x_3^{-1}x_2^{-1}x_1^{-1}$\\

 \hline
  17 & $x_1^{-1}x_3^{-1}x_2^{-1}x_3^{-1}(x_2x_3)^2(x_1x_2x_3)^{-\beta},\ $ $x_3x_2x_3x_1x_3^{-1}x_2^{-1}x_3^{-1}(x_2x_3)^2(x_1x_2x_3)^{-\beta},\ $ $x_3^{-1}(x_2x_3)^2(x_1x_2x_3)^{-\beta}$\\

 \hline
  18 & $x_1^2(x_2x_3)^2(x_2^{-1}x_3^{-1})(x_2^{-1}x_3^{-1})x_2^{-1}x_1^{-1}(x_1x_2x_3)^{-\beta},\ $ $x_2^{-2}(x_3^{-1}x_2^{-1})(x_3^{-1}x_2^{-1})x_1^{-1}(x_1x_2x_3)^{-\beta},\ $ $x_2^{-1}x_3^{-1}x_2^{-1}(x_1x_2x_3)^{-\beta}$\\

 \hline
  19 & $x_3^{-1}x_2x_3x_1x_3^{-1}x_1^{-1}x_3^{-1}x_2^{-1}x_1^{-1},\ $ $x_3x_1x_3^{-1}x_2x_3x_1x_3^{-1}x_1^{-1}x_3^{-1}x_2^{-1}x_1^{-1},\ $ $x_3^{-\beta}x_1^{-1}x_3^{-1}x_2^{-1}x_1^{-1}$\\

 \hline
  20 & $x_1(x_1x_2x_3x_2x_3^{-1}x_2^{-1})^{-\beta}(x_3^{-1}x_2^{-1}x_1^{-1})(x_3^{-1}x_2^{-1}x_1^{-1}),\ $ $x_2^{-1}x_3^{-1}x_2^{-1}(x_1x_2x_3x_2x_3^{-1}x_2^{-1})^{-\beta}(x_3^{-1}x_2^{-1}x_1^{-1})(x_3^{-1}x_2^{-1}x_1^{-1}),\ $ $x_3x_2^{-1}(x_3^{-1}x_2^{-1}x_1^{-1})(x_3^{-1}x_2^{-1}x_1^{-1})$\\

 \hline
\end{tabular}

%% file: tables/table_6.tex
\begin{tabular}{|r|p{14cm}|}
 \hline
 \textrm{No.}& $r_1,\ r_2,\ r_3$\\

 \hline
  1 & $x_1^2x_3^{-1}x_2^{-1}x_3^{-1}x_2x_3x_1^{-1},\ $ $x_2^{-1}x_3x_2x_3x_1x_3^{-1}x_2^{-1}x_3^{-1}x_2x_3x_1^{-1},\ $ $x_3^{-\epsilon}x_2x_3x_1^{-1}$\\

 \hline
  2 & $x_1^2(x_1x_3^{-1}x_2x_3)^{-\epsilon}x_3^{-1}x_2^{-1}x_1^{-1},\ $ $x_2^{-1}x_3(x_1x_3^{-1}x_2x_3)^{-\epsilon}x_3^{-1}x_2^{-1}x_1^{-1},\ $ $x_3x_1^{-1}$\\

 \hline
  3 & $x_1(x_2x_3)^2x_2^{-1}x_3^{-1}x_2^{-1},\ $ $x_2x_3^{-1}x_2^{-1},\ $ $x_3^{-\epsilon-1}x_2^{-1}x_3^{-1}x_2^{-1}x_1x_2x_3$\\

 \hline
  4 & $x_1^{-2}(x_1x_3^{-1}x_2x_3)^{-\epsilon}x_1x_2x_3,\ $ $x_2^3x_3(x_1x_3^{-1}x_2x_3)^{-\epsilon}x_1x_2x_3,\ $ $x_3^{-1}x_2x_3x_1x_2x_3$\\

 \hline
  5 & $x_1^2(x_2x_3)^2(x_2^{-1}x_3^{-1})(x_2^{-1}x_3^{-1})x_2^{-1}x_1^{-1}(x_1x_2x_3)^{-\epsilon},\ $ $(x_3^{-1}x_2^{-1})(x_3^{-1}x_2^{-1})x_1^{-1}(x_1x_2x_3)^{-\epsilon},\ $ $x_2^{-1}x_3^{-1}x_2^{-1}(x_1x_2x_3)^{-\epsilon}$\\

 \hline
  6 & $x_1^{-1}x_3^{-1}x_2^{-1}x_3^{-1}(x_2x_3)^2(x_1x_2x_3)^{-\epsilon},\ $ $x_2^2x_3x_2x_3x_1x_3^{-1}x_2^{-1}x_3^{-1}(x_2x_3)^2(x_1x_2x_3)^{-\epsilon},\ $ $x_3^{-1}(x_2x_3)^2(x_1x_2x_3)^{-\epsilon}$\\

 \hline
  7 & $x_1^2(x_2x_3)^{\epsilon}x_2(x_2x_3)^{-\epsilon}x_3^{-1}x_2^{-1}x_1^{-1},\ $ $x_2^{-2}(x_2x_3)^{-\epsilon}x_1(x_2x_3)^{\epsilon}x_2(x_2x_3)^{-\epsilon}x_3^{-1}x_2^{-1}x_1^{-1},\ $ $x_3^2(x_2x_3)^{-\epsilon}x_3^{-1}x_2^{-1}x_1^{-1}$\\

 \hline
  8 & $x_1^{-\epsilon+1}x_3^{-1}x_2x_3x_1x_3^{-1}x_1^{-1},\ $ $x_2^{-1}x_3x_1x_3^{-1}x_2x_3x_1x_3^{-1}x_1^{-1},\ $ $x_3x_1^{-1}$\\

 \hline
  9 & $x_1(x_2x_3)^2x_2^{-1}x_3^{-1}x_2^{-1},\ $ $x_2^{-\epsilon-1}x_3^{-1}x_2^{-1},\ $ $x_3x_2^{-1}x_3^{-1}x_2^{-1}x_1x_2x_3$\\

 \hline
  10 & $x_1^{-1}(x_3^{-1}x_2x_3x_1)(x_3^{-1}x_2x_3x_1)x_2x_3,\ $ $x_2^{-\epsilon+1}(x_3x_1x_3^{-1}x_2)(x_3x_1x_3^{-1}x_2)x_3x_1x_2x_3,\ $ $x_3^{-1}x_2x_3x_1x_2x_3$\\

 \hline
  11 & $x_1^{-\epsilon}(x_2x_3)^2x_2^{-1}x_3^{-1}x_2^{-1}x_1x_2x_3,\ $ $x_3^{-1}x_2^{-1}x_1x_2x_3,\ $ $x_2^{-1}x_3^{-1}x_2^{-1}(x_1x_2x_3)^2$\\

 \hline
  12 & $x_1^{-1}(x_2x_3)^{\epsilon}x_2(x_2x_3)^{-\epsilon}(x_1x_2x_3)^2,\ $ $x_2(x_2x_3)^{-\epsilon}x_1(x_2x_3)^{\epsilon}x_2(x_2x_3)^{-\epsilon}(x_1x_2x_3)^2,\ $ $x_3^{-1}(x_2x_3)^{-\epsilon}(x_1x_2x_3)^2$\\

 \hline
  13 & $(x_2x_3)^{\epsilon}x_2(x_2x_3)^{-\epsilon}x_1x_2x_3,\ $ $x_2^2(x_2x_3)^{-\epsilon}x_1(x_2x_3)^{\epsilon}x_2(x_2x_3)^{-\epsilon}x_1x_2x_3,\ $ $x_3^{-2}(x_2x_3)^{-\epsilon}x_1x_2x_3$\\

 \hline
  14 & $x_1^{-\epsilon}x_2x_3x_2^{-1}x_3^{-1}x_2^{-1}x_1^{-1}x_2x_3x_2^{-1},\ $ $x_2x_3^{-1}x_2^{-1}x_1^{-1}x_2x_3x_2^{-1},\ $ $x_2^{-1}x_1x_2x_3$\\

 \hline
  15 & $x_3^{-1}x_2^{-1}x_3^{-1}(x_2x_3)^2x_1x_2x_3,\ $ $x_2^{-\epsilon+1}x_3x_2x_3x_1x_3^{-1}x_2^{-1}x_3^{-1}(x_2x_3)^2x_1x_2x_3,\ $ $x_3^{-2}(x_2x_3)^2x_1x_2x_3$\\

 \hline
  16 & $x_1^2x_2x_3x_2^{-1}x_3^{-1}x_2^{-1}x_1^{-1}x_2x_3x_2^{-1},\ $ $x_2^{-\epsilon-1}x_3^{-1}x_2^{-1}x_1^{-1}x_2x_3x_2^{-1},\ $ $x_2^{-1}x_1x_2x_3$\\

 \hline
  17 & $x_1^{-\epsilon+1}x_3^{-1}x_2^{-1}x_3^{-1}x_2x_3x_1^{-1}x_3^{-1}x_2^{-1}x_1^{-1},\ $ $x_3x_2x_3x_1x_3^{-1}x_2^{-1}x_3^{-1}x_2x_3x_1^{-1}x_3^{-1}x_2^{-1}x_1^{-1},\ $ $x_3x_2x_3x_1^{-1}x_3^{-1}x_2^{-1}x_1^{-1}$\\

 \hline
  18 & $x_1^3(x_2x_3)^{\epsilon}x_2(x_2x_3)^{-\epsilon}(x_3^{-1}x_2^{-1}x_1^{-1})(x_3^{-1}x_2^{-1}x_1^{-1}),\ $ $x_2^{-1}(x_2x_3)^{-\epsilon}x_1(x_2x_3)^{\epsilon}x_2(x_2x_3)^{-\epsilon}(x_3^{-1}x_2^{-1}x_1^{-1})(x_3^{-1}x_2^{-1}x_1^{-1}),\ $ $x_3(x_2x_3)^{-\epsilon}(x_3^{-1}x_2^{-1}x_1^{-1})(x_3^{-1}x_2^{-1}x_1^{-1})$\\

 \hline
  19 & $x_3^{-1}x_2x_3x_1x_3^{-1}x_2x_3(x_1x_2x_3)^{-\epsilon},\ $ $x_2^2(x_3x_1x_3^{-1}x_2)(x_3x_1x_3^{-1}x_2)x_3(x_1x_2x_3)^{-\epsilon},\ $ $x_3^{-2}x_2x_3(x_1x_2x_3)^{-\epsilon}$\\

 \hline
  20 & $x_1(x_2x_3x_2^{-1}x_3^{-1}x_2^{-1}x_1^{-1})(x_2x_3x_2^{-1}x_3^{-1}x_2^{-1}x_1^{-1})(x_1x_2x_3)^{-\epsilon},\ $ $x_3^{-1}x_2^{-1}x_1^{-1}x_2x_3x_2^{-1}x_3^{-1}x_2^{-1}x_1^{-1}(x_1x_2x_3)^{-\epsilon},\ $ $x_3x_2^{-1}(x_1x_2x_3)^{-\epsilon}$\\

 \hline
 \end{tabular}